\documentclass[a4paper]{amsart}
\usepackage{mathmarco}
\usepackage[maxalphanames=99,url=false,maxbibnames=99,giveninits=true,style=alphabetic,backend=bibtex]{biblatex}
\AtEveryBibitem{
  \ifentrytype{article}{\clearfield{doi}\clearfield{issn}\clearfield{isbn}}{}
  \ifentrytype{incollection}{\clearfield{doi}\clearfield{issn}\clearfield{isbn}}{}
  \ifentrytype{book}{\clearfield{doi}\clearfield{issn}\clearfield{isbn}}{}
}

\usepackage{etoolbox}

\makeatletter

\@ifpackageloaded{hyperref}{
  \@ifpackageloaded{biblatex}{
    \DeclareRobustCommand{\SkipTocEntry}[5]{}
  }
}{}

\let\origsubsection\subsection
\renewcommand{\subsection}{
  \@ifstar{\subsection@star}{\origsubsection}%
}
\newcommand{\subsection@star}[1]{
  \addtocontents{toc}{\protect\SkipTocEntry}
  \origsubsection*{#1}
}

\let\origsection\section
\renewcommand{\section}{
  \@ifstar{\section@star}{\origsection}%
}
\newcommand{\section@star}[1]{
  \addtocontents{toc}{\protect\SkipTocEntry}
  \origsection*{#1}
}

\newcommand{\addresseshere}{
  \enddoc@text\let\enddoc@text\relax
}

\renewcommand*\env@matrix[1][\arraystretch]{%
  \edef\arraystretch{#1}%
  \hskip -\arraycolsep
  \let\@ifnextchar\new@ifnextchar
  \array{*\c@MaxMatrixCols c}}

\makeatother

\title{Optimal Hypercontractivity and Log--Sobolev inequalities on Cyclic Groups $\mathbb{Z}_{m\cdot 2^k}$}
\date{\today}

\author{Gan Yao}

\address[Gan Yao]{Institute for Advanced Study in Mathematics, Harbin Institute of Technology,  Harbin 150001, China.}
\email{gan.yao3@outlook.com}

\begin{document}

\begin{abstract}
For $1<p\le q<\infty$ and $n\in\{3\cdot 2^{k},2^{k}\}$ with $k\ge 1$, we prove that the Poisson-like semigroup $(P_t)_{t\in \mathbb{R}_+}$ on $\mathbb{Z}_n$, associated with the word length $\psi_n(k)=\min(k,n-k)$, is hypercontractive from $L_p$ to $L_q$ if and only if $t\ge \tfrac{1}{2}\log\big(\tfrac{q-1}{p-1}\big)$. We establish sharp Log--Sobolev inequalities with the optimal constant $2$, by performing a KKT analysis, and lifting from the base cases $\mathbb{Z}_6$ and $\mathbb{Z}_4$ via a Cooley--Tukey $n\mapsto 2n$ comparison of Dirichlet forms. The general case for arbitrary $n$ remains open.
\end{abstract}

\maketitle
\tableofcontents
\section{Introduction}
The hypercontractivity of the Poisson-like semigroup on the cyclic group $\mathbb{Z}_n$ is a long-standing open problem for all $n$ except $n=2$ and $n=4$. More precisely, by equipping $\mathbb{Z}_n$ with the normalized counting measure $\mu_n$, the Poisson-like semigroup $(P_t)_{t\in \mathbb{R}_+}$ is defined as the family of maps $P_t:L_\infty(\mathbb{Z}_n,\mu_n)\to L_\infty(\mathbb{Z}_n,\mu_n)$, which acts on Fourier series by
\[
  P_t:\sum_{k=0}^{n-1} a_k\chi_k(x) \mapsto \sum_{k=0}^{n-1} e^{-t\psi_n(k)}a_k\chi_k(x),
\]
where $\psi_n(k)=\min(k,n-k)$ is the word-length function on $\mathbb{Z}_n$, and $\chi_k(x)=e^{\frac{2\pi ikx}{n}}\in L_\infty(\mathbb{Z}_n)$. The hypercontractivity problem asks for the optimal time $t_{p,q}$ with $1< p\leq q<\infty$ such that
\[
    \norm{P_t f}_{q}\le \norm{f}_{p}\qquad \text{for all } t\ge t_{p,q}.
\]
For the case $n=2$, the optimal time $t_{p,q}=\tfrac{1}{2}\log\big(\tfrac{q-1}{p-1}\big)$ follows by applying the classical two-point inequality. That inequality was first proved by Bonami~\cite{MR283496}, later rediscovered by Gross~\cite{MR420249}, and was also used by Beckner~\cite{MR385456} to obtain the best constants for the Hausdorff--Young inequality. For $n=4$, Beckner, Janson, and Jerison~\cite{MR730056} obtained the same optimal time $t_{p,q}=\tfrac{1}{2}\log\big(\tfrac{q-1}{p-1}\big)$ by a clever reduction from $\mathbb{Z}_2$ to $\mathbb{Z}_4$, though the approach does not extend to other $\mathbb{Z}_{m\times n}$. For $n=3$, by Wolff's reduction \cite[Corollary 3.1]{MR2314078}, the optimal time $t_{2,q}$ of $P_t$ on $\mathbb{Z}_3$ coincides the optimal time 
\[
t_{2,q}=\frac{1}{2}\log(\frac{\frac{2}{3}(\frac{1}{3})^{\frac{2}{q}-1}-\frac{1}{3}(\frac{2}{3})^{\frac{2}{q}-1}}{(\frac{2}{3})^{\frac{2}{q}}-(\frac{1}{3})^{\frac{2}{q}}})
\]
of the simple semigroup $T_t=e^{-t(\mathrm{id}-\mathbb{E}_{\delta})}$ on the weighted two-point space $L_\infty(\{-1,1\},\tfrac{1}{3}\delta_{-1}+\tfrac{2}{3}\delta_{1})$, which was computed by Oleszkiewicz and Lata\l a \cite{MR1796718,MR2073432}. For the general optimal time $t_{p,q}$ of $(P_t)_{t\in \mathbb{R}_+}$ on $\mathbb{Z}_3$, only asymptotic information is available; see \cite[Theorem 2.1]{MR2314078}. For $n=5$, Andersson~\cite{MR1883499} determined $t_{2,q}=\tfrac{1}{2}\log(q-1)$ for $q\in 2\mathbb{Z}_+$. For $n\ge 6$, Junge, Palazuelos, Parcet, and Perrin~\cite{MR3709719} proved partial results: for even $n$, $t_{2,q}=\tfrac{1}{2}\log(q-1)$ for $q\in 2\mathbb{Z}_+$; for odd $n$, the same holds when $n\ge q$, via a rather involved combinatorial method. Combining Stein's interpolation method \cite{MR82586} with Gross's extrapolation technique \cite{MR420249,MR372613,MR2325763}, they further obtained $t_{p,q}\le \tfrac{\log 3}{2}\log\big(\frac{q-1}{p-1}\big)$ in those regimes.\par

Hypercontractivity is also widely studied for semigroups on manifolds. Among the most classical examples, Weissler~\cite{MR578933} proved optimal hypercontractivity for the heat and Poisson semigroups on the circle $\mathbb{S}^1$, while Rothaus~\cite{MR593787} independently treated the heat case. For the sphere $\mathbb{S}^n$ ($n\ge 2$), Mueller and Weissler~\cite{MR674060} established optimal hypercontractivity for the heat semigroup. For Poisson-type semigroups on $\mathbb{S}^n$, see Janson~\cite{MR706641}, Beckner~\cite{MR1164616}, and Frank--Ivanisvili~\cite{MR4278128}.\par 

Now we present our main result for the Poisson-like semigroup $(P_t)_{t\in\mathbb{R}_+}$ on $\mathbb{Z}_n$. Note that the standard argument, based on a simple series expansion of $\|P_t(1+\epsilon f)\|_q$ and $\|1+\epsilon f\|_p$ at $\epsilon=0$, shows the universal lower bound $t_{p,q}\ge \tfrac{1}{2}\log\big(\tfrac{q-1}{p-1}\big)$. The theorem below records the optimal times $t_{p,q}$ along a dyadic tower of $n$.
\begin{theorem}\label{Thm: Hypercontractivity on tower of n}
  For $n=3\cdot 2^k$ and $n=2^k$ with $k\ge 1$, we have 
  \[
    \norm{P_{t}f}_q\le \norm{f}_p\quad \Leftrightarrow \quad t\ge \frac{1}{2}\log(\frac{q-1}{p-1})
    \]
    for $1<p\le q<\infty$.
\end{theorem}
A standard route to hypercontractivity proceeds through Log--Sobolev inequalities (LSI) by Gross's celebrated work~\cite{MR420249}: $\norm{P_{t}f}_q\le \norm{f}_p$ holds whenever $t\geq \frac{C}{4}\log\big(\frac{q-1}{p-1}\big)$ if and only if the corresponding LSI holds with constant $C$. Thus, to prove Theorem~\ref{Thm: Hypercontractivity on tower of n}, it suffices to establish the following $n$--LSI with the optimal constant $2$ along the above dyadic tower of $n$. Denote by $A_{\psi_n}$ the generator of semigroup $(P_t)_{t\in\mathbb{R}_+}$, that is
\begin{equation}
  A_{\psi_n}: \sum_{k=0}^{n-1} a_k\chi_k(x)\mapsto \sum_{k=0}^{n-1} \psi_n(k)a_k\chi_k(x).
\end{equation}
 The case $n=4$ in the following theorem is due to the work \cite{MR730056} and Gross's extrapolation technique~\cite{MR420249}.
\begin{theorem}\label{Thm: Log Sobolev inequality n=6 times 2^k and n=8 times 2^k}
For $n=3\cdot 2^k$ and $n=2^k$ with $k\ge 1$, we have the following LSI with the optimal constant $2$:
\[
  \int_{\mathbb{Z}_n} f^2\log f^2\dd\mu_n-\norm{f}_2^2\log\norm{f}_2^2\le 2\inner{f,A_{\psi_n}f}_{L_2(\mathbb{Z}_n,\mu_n)},\qquad f\in L_2^+(\mathbb{Z}_n,\mu_n).
\]
\end{theorem}

Our proof of Theorem~\ref{Thm: Log Sobolev inequality n=6 times 2^k and n=8 times 2^k} is based on a new induction scheme with three key ingredients:
\begin{enumerate}
  \item Auxiliary weights $\phi_4$ on $\mathbb{Z}_4$ and $\phi_6$ on $\mathbb{Z}_6$ together with their corresponding LSIs. These LSIs are tighter than those for the word-lengths $\psi_n$; this refinement is new even when $n=4$ and plays a key role in the analysis of LSIs.
  \item Karush--Kuhn--Tucker (KKT) analysis for the $4$-- and $6$--LSI.  
We develop an efficient way to handle LSIs on $\mathbb{Z}_n$ via KKT analysis, combined with the aforementioned manipulation of the length functions. The specific structure of $\phi_6$ introduces a symmetry in the KKT system that makes the analysis tractable.
  \item An induction from the $n$--LSI to the $2n$--LSI under a crucial compatibility condition. We use a Cooley--Tukey factorization of the $2n$-point discrete Fourier transform (DFT), which expresses a large DFT as a combination of smaller DFTs and yields a comparison of Dirichlet forms at the scales $n$ and $2n$. The choice of the new weights $\phi_4$ and $\phi_6$ mentioned in (1) is crucial for the base steps $4\to 8$ and $6\to 12$.
\end{enumerate}
Finally, we remark that the above ideas are also useful for studying LSIs along other towers of the form $m\cdot n^k$. The KKT analysis and the compatibility condition vary in their technical details, depending on the specific value of $n$ and on the particular choice of weights on $\mathbb{Z}_n$, and we will not carry out this analysis in this paper.\par

The article is organized as follows. After a brief introduction to the LSI formulation and the $n$-dimensional KKT framework in Section~\ref{Sec: DFT and KKT analysis}, we analyze the KKT systems associated with the LSIs for $n=6$ and $n=4$ in Section~\ref{Sec: Log Sobolev inequality n=4 and n=6}. Section~\ref{Sec: Induction from n to 2n} presents a comparison criterion for a pair of weights that allows us to compare the Dirichlet forms at the scales $n$ and $2n$. We then apply it to the transition $n\to 2n$ of LSIs and establish Theorem~\ref{Thm: Log Sobolev inequality n=6 times 2^k and n=8 times 2^k}.

\section{Fourier formulation of LSI on $\mathbb{Z}_n$ and KKT framework}\label{Sec: DFT and KKT analysis}

Let $F_n$ denote the $n\times n$ discrete Fourier transform (DFT) matrix
\[
    F_n=\begin{pmatrix}
        1 & 1 & 1 & \cdots & 1\\
        1 & \omega & \omega^2 & \cdots & \omega^{n-1}\\
        1 & \omega^2 & \omega^4 & \cdots & \omega^{2(n-1)}\\
        \vdots & \vdots & \vdots & \ddots & \vdots\\
        1 & \omega^{n-1} & \omega^{2(n-1)} & \cdots & \omega^{(n-1)(n-1)}
    \end{pmatrix},
\]
where $\omega=e^{2\pi i/n}$ and, with $0$-indexed rows/columns, $(F_n)_{j,k}=\omega^{jk}$ for $0\le j,k\le n-1$. It is well known that $\frac{1}{\sqrt{n}}F_n$ is unitary (see, e.g., \cite[Section 2.5]{MR543191}). For a column vector $x\in\mathbb{C}^n$, we write its DFT as
\[
      \hat x=(\hat{x}_0,\ldots,\hat{x}_{n-1})^{\mathrm{T}}=F_nx,\quad \text{where}\quad \hat{x}_k=\sum_{j=0}^{n-1}x_j\omega^{jk}.
\]\par 

Given $\lambda\in\mathbb{R}_+^n$, define the entropy functional on the vector $\lambda$ by
\[
  \mathrm{H}_n[\lambda] \coloneqq\frac{1}{n}\sum_{k=0}^{n-1}\lambda_k^2\log(\lambda_k^2)-\frac{\norm{\lambda}_2^2}{n}\log(\frac{\norm{\lambda}_2^2}{n})=\frac{1}{n}\sum_{k=0}^{n-1}\lambda_k^2\log(\frac{n\lambda_k^2}{\norm{\lambda}_2^2}).
\]
 The functional $\mathrm{H}_n[\lambda]$ is homogeneous of degree $2$, i.e., $\mathrm{H}_n[\alpha\lambda]=\alpha^2\mathrm{H}_n[\lambda]$ for $\alpha>0$.\par

The unitary $\frac{1}{\sqrt{n}}F_n$ yields the following equivalence between the $n$--LSI and the $n$-variable entropy--Dirichlet form inequality used in the paper; throughout, “LSI” refers to either formulation.

\begin{lemma}\label{Lem: Explicit form of Log Sobolev ineq}
Let $\gamma:\mathbb{Z}_n\to\mathbb{R}_+$ and write
\[
  \diag(\gamma)\coloneqq \diag\big(\gamma(0),\gamma(1),\ldots,\gamma(n-1)\big)\in M_n(\mathbb{R}).
\]
Define $A_{\gamma}:L_\infty(\mathbb{Z}_n,\mu_n)\to L_\infty(\mathbb{Z}_n,\mu_n)$ by
\[
  A_{\gamma}: \sum_{k=0}^{n-1}a_k\chi_k(x)\mapsto\sum_{k=0}^{n-1}\gamma(k)a_k\chi_k(x).
\]
Then
\[
  \int_{\mathbb{Z}_n} f^2\log f^2\dd \mu_n-\norm{f}_2^2\log\norm{f}_2^2 \le 2\inner{f,A_{\gamma}f}_{L_2(\mathbb{Z}_n,\mu_n)},\qquad f\in L_2^+(\mathbb{Z}_n,\mu_n),
\]
is equivalent to
\[
  \mathrm{H}_n[\lambda] \le 2\inner{\lambda,\Gamma\lambda},\qquad \lambda\in\mathbb{R}_+^n,
\]
where $\inner{\cdot,\cdot}$ denotes the Hermitian inner product on $\ell_2^n$ and $\Gamma=\frac{1}{n}F_n\diag(\gamma)F_n^{-1}\in M_n(\mathbb{C})$. In particular, if $\gamma(k)=\gamma(n-k)$ for $1\le k\le n-1$, then $\Gamma$ is real symmetric.
\end{lemma}

\begin{proof}
    Write $f(x)=\sum_{k=0}^{n-1} a_k\chi_k(x)$ and set $a_f\coloneqq (a_0,\ldots,a_{n-1})^{\mathrm T}$, $\lambda\coloneqq \big(f(0),\ldots,f(n-1)\big)^{\mathrm T}\in\mathbb{R}_+^n$. The verification of $\inner{f,A_{\gamma}f}_{L_2(\mathbb{Z}_n,\mu_n)}=\inner{\lambda,\Gamma\lambda}$ is straightforward by the fact the $\frac{1}{\sqrt{n}}F_n$ is  unitary and $\lambda=F_na_f$. 
Moreover, for $\lambda=\big(f(0),\ldots,f(n-1)\big)^{\mathrm T}$, we have 
\[
\int_{\mathbb{Z}_n} f^2\log f^2\dd\mu_n-\norm{f}_2^2\log\norm{f}_2^2=\frac{1}{n}\sum_{k=0}^{n-1}\lambda_k^2\log(\lambda_k^2)-\frac{\norm{\lambda}_2^2}{n}\log(\frac{\norm{\lambda}_2^2}{n})=\mathrm{H}_n[\lambda].
\]\par

Now we prove that $\Gamma$ is real symmetric if $\gamma(k)=\gamma(n-k)$ for $1\le k\le n-1$. We write $A^{\mathrm H}\coloneqq \overline{A}^{\mathrm T}$ for the Hermitian (conjugate) transpose of a complex matrix $A$. Since $\diag(\gamma)\in M_n(\mathbb{R})$ and $F_n^{-1}=\tfrac{1}{n}F_n^{\mathrm H}$, we have
\[
  \left( F_n\diag(\gamma)F_n^{-1} \right)^{\mathrm H}
  =(F_n^{-1})^{\mathrm H}\diag(\gamma)F_n^{\mathrm H}
  =F_n\diag(\gamma)F_n^{-1},
\]
hence $\Gamma=\tfrac{1}{n}F_n\diag(\gamma)F_n^{-1}$ is Hermitian.

Let $P$ denote the permutation matrix that interchanges the entries $k$ and $n-k$ for $1\le k\le n-1$. If $\gamma(k)=\gamma(n-k)$ for all such $k$, then $\diag(\gamma)$ is invariant under conjugation by $P$. Because $\overline{F_n}=F_n^{\mathrm{H}}=F_nP$ and $P=P^{-1}=P^{\mathrm{T}}$, it follows that 
\[
  \overline{F_n\diag(\gamma)F_n^{-1}}=\overline{F_n}\diag(\gamma)\overline{F_n^{-1}}=F_n P\diag(\gamma)PF_n^{-1}=F_n\diag(\gamma)F_n^{-1}.
\]
Hence $\Gamma\in M_n(\mathbb{R})$ and is symmetric, completing the proof.
\end{proof}

\subsection*{Homogeneous objective and reduction to $\mathbb{S}^{n-1}_+$.}
Given a length function $\psi_n$, define
\[
    f_{\psi_n}(\lambda)\coloneqq 2\inner{\lambda,\Psi(n)\lambda}-\mathrm{H}_n[\lambda],
\]
where $\lambda=(\lambda_0,\ldots,\lambda_{n-1})^{\mathrm{T}}$, $\Psi(n)=\frac{1}{n}F_n\diag(\psi_n)F_n^{-1}$, and we set $x\log(x^2)=0$ at $x=0$. Since $f_{\psi_n}$ is homogeneous of degree $2$, it suffices to verify $f_{\psi_n}\ge 0$ on the positive sphere
\[
    \mathbb{S}^{n-1}_+ \coloneqq \Big\{ \lambda\in\mathbb{R}_+^n : \sum_{k=0}^{n-1}\lambda_k^2=1 \Big\}
\]
in order to conclude $f_{\psi_n}\ge 0$ on $\mathbb{R}^n_+$. We therefore restrict $f_{\psi_n}$ on $\mathbb{S}^{n-1}_+$ and analyze stationary points via the Karush--Kuhn--Tucker (KKT) conditions (see the original sources \cite{MR2936770,MR47303} and textbook treatments \cite{MR3587371,MR2244940}). In brief, the KKT conditions provide necessary first-order conditions for constrained optimization problems with both equality and inequality constraints, extending the classical Lagrange multiplier method by introducing one multiplier for each active inequality constraint, together with the complementary slackness. 
\begin{proposition}[KKT necessary conditions {\cite[Proposition 4.3.1]{MR3587371}, \cite[Theorem 12.1]{MR2244940}}]\label{Prop: KKT conditions}
  Let $\mathcal{E}$ and $\mathcal{I}$ be finite index sets, $g\in C^1(\mathbb{R}^n)$, and $c_i\in C^1(\mathbb{R}^n)$ for $i\in\mathcal{E}\cup \mathcal{I}$. Suppose that $\lambda^*\in \mathbb{R}^n$ is a local minimizer of $g(\lambda)$ subject to the constraints
  \[
      \begin{cases}
        c_i(\lambda)=0,\qquad i\in \mathcal{E},\\
        c_i(\lambda)\ge 0,\qquad i\in \mathcal{I},
      \end{cases}
  \]
  Define the Lagrangian 
  \[
      \mathcal{L}(\lambda,\mu,\nu)=g(\lambda)-\sum_{i\in \mathcal{E}}\mu_i c_i(\lambda)-\sum_{i\in \mathcal{I}}\nu_ic_i(\lambda),\qquad \lambda\in \mathbb{R}^n, \mu\in \mathbb{R}^{\abs{\mathcal{E}}}, \nu\in \mathbb{R}^{\abs{\mathcal{I}}}.
  \]
  If $\lambda^*$ is regular, i.e., the vector set
  \[
      \{\grad c_i(\lambda^*)\}_{i\in \mathcal{E}}\cup\{\grad c_i(\lambda^*)\}_{i\in A(\lambda^*)}
  \]
  are linearly independent, where $A(\lambda^*)\coloneqq\{i\in\mathcal I: c_i(\lambda^*)=0\}$, then there exist unique Lagrange multiplier vectors $\mu^*\in \mathbb{R}^{\abs{\mathcal{E}}}$, $\nu^*\in \mathbb{R}^{\abs{\mathcal{I}}}$ such that 
  \[
      \begin{cases}
        \grad_\lambda\mathcal{L}(\lambda^*,\mu^*,\nu^*)=0,\\
        c_i(\lambda^*)=0,\qquad i\in \mathcal{E},\\
        c_i(\lambda^*)\ge 0,\qquad i\in \mathcal{I},\\
        \nu_i^*c_i(\lambda^*)=0,\qquad i\in \mathcal{I},\\
        \nu_i^*\geq 0,\qquad i\in \mathcal{I}.
      \end{cases}
  \]  
\end{proposition}\par

We now apply Proposition~\ref{Prop: KKT conditions} to the minimization of the particular entropy functional relevant to our problem on $\mathbb{S}^{n-1}_+$. We can absorb the multiplier $\mu$ into a normalization.

\begin{lemma}\label{Lem: KKT condition absorbing mu}
Let $Q\in M_n(\mathbb{R})$ be a symmetric matrix and set
\[
    g(\lambda)=2\inner{\lambda,Q\lambda}-\mathrm{H}_n[\lambda].
\]
If the system
\begin{equation}\label{Eqns: KKT condition absorbing mu}
\begin{cases}
4Q\lambda-\dfrac{4}{n}\begin{pmatrix}
\lambda_0\log(\lambda_0)\\
\vdots\\
\lambda_{n-1}\log(\lambda_{n-1})
\end{pmatrix}-\nu=0,\\
0<\norm{\lambda}_2^2<n,\\
\lambda_j\geq 0,\qquad 0\le j\le n-1,\\
\lambda_j\nu_j =0,\qquad 0\le j\le n-1,\\
\nu_j\ge 0,\qquad 0\le j\le n-1,
\end{cases}
\end{equation}
has no solution, then $g\ge 0$ on $\mathbb{S}^{n-1}_+$.
\end{lemma}

\begin{proof}
  Consider the Lagrangian associated with the minimization of $g$ on $\mathbb{S}^{n-1}_+$:
\[
  \mathcal{L}(\lambda,\mu,\nu)= g(\lambda)-\mu\left(\sum_{k=0}^{n-1}\lambda_k^2-1\right)-\sum_{k=0}^{n-1}\nu_k\lambda_k
\]
with multipliers $\mu\in\mathbb{R}$ and $\nu\in\mathbb{R}^n_{+}$ associated to the constraints $\norm{\lambda}_2=1$ and $\lambda_j\ge 0$. Notice the set 
\[
  \{2\lambda\}\cup\{e_j\}_{j\in A(\lambda)}
\]
are linear independent, where $(e_j)_{0\leq j\leq n-1}$ denotes the canonical basis in $\mathbb{R}^n$, and $A(\lambda)=\{j:\lambda_j=0\}$. Hence, by Proposition~\ref{Prop: KKT conditions}, for any local minimizer $\lambda^*$, there exist unique $\mu^*\in\mathbb{R}$ and $\nu^*\in\mathbb{R}^n_{+}$ such that the following system holds:
\begin{equation}\label{Eqns: minimum of fn geq 0}
\begin{cases}
\grad_{\lambda} \mathcal{L}
= 4Q\lambda^*-\dfrac{2}{n}\begin{pmatrix}
2\lambda^*_0\log(\lambda^*_0)+\lambda^*_0\\
\vdots\\
2\lambda^*_{n-1}\log(\lambda^*_{n-1})+\lambda^*_{n-1}
\end{pmatrix}+\dfrac{2}{n}\begin{pmatrix}
\lambda^*_0\log(\frac{\norm{\lambda^*}_2^2}{n})+\lambda^*_0\\
\vdots\\
\lambda^*_{n-1}\log(\frac{\norm{\lambda^*}_2^2}{n})+\lambda^*_{n-1}
\end{pmatrix}
-2\mu^*\lambda^*-\nu^*=0,\\
\norm{\lambda^*}_2^2=1,\\
\lambda^*_j\geq 0,\qquad 0\le j\le n-1,\\
\lambda^*_j\nu^*_j =0,\qquad 0\le j\le n-1,\\
\nu^*_j\ge 0,\qquad 0\le j\le n-1.
\end{cases}
\end{equation}
Write
\[
\begin{aligned}
g(\lambda)
&=\mathcal{L}(\lambda,\mu,\nu)
   +\mu\left(\sum_{k=0}^{n-1}\lambda_k^2-1\right)
   +\sum_{k=0}^{n-1}\nu_k\lambda_k.
\end{aligned}
\]
Using Euler's theorem for homogeneous functions of degree $2$ (namely $\inner{\grad g(\lambda^*),\lambda^*}=2g(\lambda^*)$), together with
$\grad_{\lambda}\mathcal{L}(\lambda^*,\mu^*,\nu^*)=0$, $\norm{\lambda^*}_2^2=1$, and the fourth line of \eqref{Eqns: minimum of fn geq 0}, we obtain
\[
    2g(\lambda^*)=2\mu^*.
\]\par 

Assume that there exists a stationary point $\lambda^*$ with $g(\lambda^*)<0$, then $\mu^*=g(\lambda^*)<0$. Let $c^*=e^{\frac{n\mu^*+\log n}{2}}<\sqrt{n}$. So we have
\[
\begin{aligned}
0
&=4Q(c^*\lambda^*)-\frac{4}{n}
\begin{pmatrix}
c^*\lambda^*_0\log(c^*\lambda^*_0)-c^*\lambda^*_0\log c^*\\
\vdots\\
c^*\lambda^*_{n-1}\log(c^*\lambda^*_{n-1})-c^*\lambda^*_{n-1}\log c^*
\end{pmatrix}
-\frac{2}{n}c^*\lambda^*\\
&\quad+\frac{2}{n}\begin{pmatrix}
  c^*\lambda^*_0\log(\frac{1}{n})+c^*\lambda^*_0\\
  \vdots\\
  c^*\lambda^*_{n-1}\log(\frac{1}{n})+c^*\lambda^*_{n-1}
\end{pmatrix}-2\mu^*c^*\lambda^*-c^*\nu^*\\
&=4Q(c^*\lambda^*)-\frac{4}{n}
\begin{pmatrix}
c^*\lambda^*_0\log(c^*\lambda^*_0)\\
\vdots\\
c^*\lambda^*_{n-1}\log(c^*\lambda^*_{n-1})
\end{pmatrix}
+\Big(\frac{4}{n}\log c^*-\frac{2}{n}+\frac{2}{n}(-\log(n)+1)-2\mu^*\Big)c^*\lambda^*-c^*\nu^*\\
&=4Q(c^*\lambda^*)-\frac{4}{n}
\begin{pmatrix}
c^*\lambda^*_0\log(c^*\lambda^*_0)\\
\vdots\\
c^*\lambda^*_{n-1}\log(c^*\lambda^*_{n-1})
\end{pmatrix}
-c^*\nu^*,
\end{aligned}
\]
and $0<\norm{c^*\lambda^*}_2^2<n$. Thus $(c^*\lambda^*,c^*\nu^*)$ solves \eqref{Eqns: KKT condition absorbing mu}, contradicting the assumption. Therefore no such $\lambda^*$ exists and $g\ge 0$ on $ \mathbb{S}^{n-1}_+$.
\end{proof}

\section{LSI on $\mathbb{Z}_4$ and $\mathbb{Z}_6$ with modified weights}\label{Sec: Log Sobolev inequality n=4 and n=6}
Building on the KKT framework from Section~\ref{Sec: DFT and KKT analysis}, we now establish the LSI on $\mathbb{Z}_4$ and $\mathbb{Z}_6$ with tighter estimates. To this end, we introduce suitably modified weight functions, which are crucial both for making the resulting KKT systems tractable and for ensuring that the induction procedure in the next section applies.

\begin{theorem}\label{Thm: Log Sobolev inequality n=4}
For the weight function
\[
  \phi_4(j)=
  \begin{cases}
    \psi_4(j),& j\neq 2,\\
    \frac{8}{5},& j=2,
  \end{cases}
\]
on $\mathbb{Z}_4$, we have
\[
  \int_{\mathbb{Z}_4} f^2\log f^2\dd\mu_4-\norm{f}_2^2\log\norm{f}_2^2\le 2\inner{f,A_{\phi_4} f}_{L_2(\mathbb{Z}_4,\mu_4)}
\]
for all $f\in L_2^+(\mathbb{Z}_4,\mu_4)$.
\end{theorem}

\begin{theorem}\label{Thm: Log Sobolev inequality n=6}
For the weight function
\[
  \phi_6(j)=
  \begin{cases}
    \psi_6(j),& j\neq 3,\\
    1,& j=3,
  \end{cases}
\]
on $\mathbb{Z}_6$, we have
\[
  \int_{\mathbb{Z}_6} f^2\log f^2\dd\mu_6-\norm{f}_2^2\log\norm{f}_2^2\le 2\inner{f,A_{\phi_6} f}_{L_2(\mathbb{Z}_6,\mu_6)}
\]
for all $f\in L_2^+(\mathbb{Z}_6,\mu_6)$.
\end{theorem}

Since $\phi_n\le \psi_n$ pointwise, we have $\inner{f,A_{\phi_n}f}\leq \inner{f,A_{\psi_n}f}$ for $f\in L_2^+(\mathbb{Z}_n,\mu_n)$. Therefore, the two theorems above yield LSI for $\psi_4$ and $\psi_6$. Although the $4$--LSI with the length function $\psi_4(j)=\min(j,4-j)$ can be deduced from the known hypercontractivity of $(P_t)_{t\in\mathbb{R}_+}$ on $\mathbb{Z}_4$, we show that $4$--LSI still holds for the smaller weight $\phi_4$. This choice is advantageous for the induction in the next section: the smaller weight $\phi_4$ is suitable for the base step from $4$--LSI to $8$--LSI in the $n\to 2n$ comparison, so that the LSIs can be inductively obtained for all $n=8\cdot 2^k$. For the choice of $\phi_6$, besides serving as the base step from $6$--LSI to $12$--LSI, it also introduces a symmetry in the nonlinear KKT system, which makes the analysis possible in the proof below.

We first prove Theorem~\ref{Thm: Log Sobolev inequality n=6}. The proof of Theorem~\ref{Thm: Log Sobolev inequality n=4} follows the same strategy route. Set 
\[
\Phi(6)=\frac{1}{6}F_6\diag(\phi_6)F_6^{-1}
=\begin{pmatrix}[1.2]
\frac{7}{36}  & -\frac{1}{18} & -\frac{1}{18} & \frac{1}{36}  & -\frac{1}{18} & -\frac{1}{18} \\
-\frac{1}{18} & \frac{7}{36}  & -\frac{1}{18} & -\frac{1}{18} & \frac{1}{36}  & -\frac{1}{18} \\
-\frac{1}{18} & -\frac{1}{18} & \frac{7}{36}  & -\frac{1}{18} & -\frac{1}{18} & \frac{1}{36}  \\
\frac{1}{36}  & -\frac{1}{18} & -\frac{1}{18} & \frac{7}{36}  & -\frac{1}{18} & -\frac{1}{18} \\
-\frac{1}{18} & \frac{1}{36}  & -\frac{1}{18} & -\frac{1}{18} & \frac{7}{36}  & -\frac{1}{18} \\
-\frac{1}{18} & -\frac{1}{18} & \frac{1}{36}  & -\frac{1}{18} & -\frac{1}{18} & \frac{7}{36}
\end{pmatrix}.
\]
Define
\[
    f_{\phi_6}(\lambda)
    = 2\inner{\lambda,\Phi(6)\lambda}-\mathrm{H}_6[\lambda].
\]
By Lemma~\ref{Lem: KKT condition absorbing mu}, $f_{\phi_6}\ge 0$ on $\mathbb{R}_+^6$ follows if the system below has no solution:
\begin{equation}\label{Eqns: Lagrange for n=6, (0,1,2,1,2,1)}
\begin{cases}
-\frac{2}{3}\lambda_j\log(\lambda_j)
  +\Big(\frac{7}{9}\lambda_j-\frac{2}{9}\sum_{k\neq j,j+3}\lambda_k+\frac{1}{9}\lambda_{j+3}\Big)
  -\nu_j=0,\qquad 0\le j\le 5,\\
0<\sum_{k=0}^5 \lambda_k^2 <6,\\
\lambda_j\ge 0,\qquad 0\le j\le 5,\\
\lambda_j\nu_j=0,\qquad 0\le j\le 5,\\
\nu_j\ge 0,\qquad 0\le j\le 5,
\end{cases}
\end{equation}
where indices are understood modulo $6$.\par

We divide the analysis of \eqref{Eqns: Lagrange for n=6, (0,1,2,1,2,1)} into two cases.

\begin{lemma}\label{Lem: pairwise-opposite-equality, n=6}
The system \eqref{Eqns: Lagrange for n=6, (0,1,2,1,2,1)} has no solution in the region 
\[
  \{(\lambda,\nu)\in R^{6}_+\times R^{6}_+: \forall j\in\{0,1,2\}, \lambda_{j}=\lambda_{j+3}\}.
\]
\end{lemma}

\begin{proof} 
Assume $(\lambda,\nu)$ is a solution in the above region. From the first line of
\eqref{Eqns: Lagrange for n=6, (0,1,2,1,2,1)} we get $\nu_j=\nu_{j+3}$ and, grouping opposite indices, the system \eqref{Eqns: Lagrange for n=6, (0,1,2,1,2,1)} reduces to
\begin{equation}\label{Eqns: Lagrange for n=6, (0,1,2,1,2,1), opposite equal with nu_j}
  \begin{cases}
-\frac{2}{3}\lambda_j\log(\lambda_j)+\frac{8}{9}\lambda_j-\nu_j-\frac{4}{9}(\lambda_{j+1}+\lambda_{j+2})=0,\qquad 0\leq j\leq 2,\\
0<\sum_{k=0}^2\lambda_k^2<3,\\
\lambda_j\ge 0,\qquad 0\le j\le 2,\\
\nu_j\lambda_j=0,\qquad 0\le j\le 2,\\
\nu_j\ge 0,\qquad 0\le j\le 2,
\end{cases}
\end{equation}
where indices are understood modulo $3$.\par

If $\lambda_\ell=0$ for some $\ell\in\{0,1,2\}$, taking $j=\ell$ in the first line of \eqref{Eqns: Lagrange for n=6, (0,1,2,1,2,1)} gives $\frac{4}{9}(\lambda_{\ell+1}+\lambda_{\ell+2})=-\nu_\ell\le 0$, hence $\lambda_{\ell+1}=\lambda_{\ell+2}=0$, contradicting $\sum_{k=0}^2\lambda_k^2>0$. Thus $\lambda_j>0$ for all $j$, and from $\nu_j\lambda_j=0$ in the fourth line of \eqref{Eqns: Lagrange for n=6, (0,1,2,1,2,1), opposite equal with nu_j} we get $\nu_j=0$ for $0\le j\le 2$. Therefore, we have
\begin{equation}\label{Eqns: Lagrange for n=6, (0,1,2,1,2,1), opposite equal}
  \begin{cases}
-\frac{2}{3}\lambda_j\log(\lambda_j)+\frac{4}{3}\lambda_j-\frac{4}{9}(\lambda_j+\lambda_{j+1}+\lambda_{j+2})=0,\qquad 0\le j\le 2,\\
0<\sum_{k=0}^2\lambda_k^2<3,\\
 \lambda_j>0,\qquad  0\le j\le 2.
\end{cases}
\end{equation}
Since the function $x\mapsto -\frac{2}{3}x\log x+\frac{4}{3}x$ is strictly increasing on $(0,e)$, from the first line of \eqref{Eqns: Lagrange for n=6, (0,1,2,1,2,1), opposite equal} we conclude $\lambda_0=\lambda_1=\lambda_2$, and in particular $0<\lambda_j<1$. Substituting into the first line of \eqref{Eqns: Lagrange for n=6, (0,1,2,1,2,1), opposite equal} gives
\[
-\frac{2}{3}\lambda_j\log(\lambda_j)+\frac{4}{3}\lambda_j-\frac{4}{9}(\lambda_j+\lambda_{j+1}+\lambda_{j+2})=-\frac{2}{3}\lambda_j\log(\lambda_j)> 0.
\]
Hence no solution exists.
\end{proof}

\begin{lemma}\label{Lem: one opposite pair unequal, n=6}
The system \eqref{Eqns: Lagrange for n=6, (0,1,2,1,2,1)} has no solution in the region 
\[
  \{(\lambda,\nu)\in R^{6}_+\times R^{6}_+: \exists j_0\in\{0,1,2\}, \lambda_{j_0}\ne \lambda_{j_0+3}\}.
\]
\end{lemma}

\begin{proof}
Assume $(\lambda,\nu)$ is a solution in the above region. We divide the proof into four steps.\par 
  \textbf{Step 1.} We show that $\lambda_j>0$ for all $j\in\{0,1,2,3,4,5\}$ and hence $\nu_j=0$ for all such $j$. Take $\ell\in \{0,1,2,3,4,5\}$.\par
\begin{itemize}
  \item \emph{Case 1: $\lambda_{\ell}=\lambda_{\ell+3}$.}
  Arguing by contradiction, assume that $\lambda_{\ell}=\lambda_{\ell+3}=0$ for some $\ell\in\{0,1,2\}$. 
  Then, from the first line in \eqref{Eqns: Lagrange for n=6, (0,1,2,1,2,1)} with index $\ell$, we obtain
  \[
    \frac{2}{9}\sum_{k\neq \ell,\ell+3}\lambda_k=-\nu_{\ell}\le 0.
  \]
Therefore, by $\lambda_j\geq 0$, we have $\lambda_j=0$ for all $j$, which contradicts $0<\sum_{k=0}^5\lambda_k^2$. Hence $\lambda_{\ell}>0$ and by $\lambda_{\ell}\nu_{\ell}=0$ in \eqref{Eqns: Lagrange for n=6, (0,1,2,1,2,1)} we have $\nu_{\ell}=0$ for such index $\ell$.\par 
    \item \emph{Case 2: $\lambda_{\ell}\neq \lambda_{\ell+3}$.} For $0\le j\le 2$, the first line of \eqref{Eqns: Lagrange for n=6, (0,1,2,1,2,1)} can be equivalently written as
\begin{equation}\label{eq:pair-form}
  -\frac{2}{3}\lambda_j\log(\lambda_j)+\frac{2}{3}\lambda_j-\nu_j
  =\frac{2}{9}\sum_{k=0}^5\lambda_k-\frac{1}{3}\lambda_j-\frac{1}{3}\lambda_{j+3},
  \qquad 0\leq j\leq 2.
\end{equation}
Let $F(x)=-\frac{2}{3}x\log x+\frac{2}{3}x$. 
From the symmetry between the equations for indices $j$ and $j+3$ in the right hand side of \eqref{eq:pair-form}, we obtain
\begin{equation}\label{eq: opposite F}
  F(\lambda_j)-\nu_j = F(\lambda_{j+3})-\nu_{j+3}.
\end{equation}
Assume that $0=\lambda_{\ell}<\lambda_{\ell+3}$. Then $F(\lambda_{\ell})=F(0)=0$. Moreover, since $\lambda_{\ell+3}>0$ and $\lambda_{\ell+3}\nu_{\ell+3}=0$ in \eqref{Eqns: Lagrange for n=6, (0,1,2,1,2,1)}, we have $\nu_{\ell+3}=0$. Hence \eqref{Eqns: Lagrange for n=6, (0,1,2,1,2,1)} and \eqref{eq: opposite F} yields
\[
  0\ge -\nu_{\ell}=F(\lambda_{\ell+3})
  = -\frac{2}{3}\lambda_{\ell+3}\log(\lambda_{\ell+3})+\frac{2}{3}\lambda_{\ell+3}.
\]
So $\lambda_{\ell+3}\ge e>\sqrt{6}$, contradicting $0<\sum_{k=0}^5\lambda_k^2<6$. Therefore, whenever $\lambda_{\ell}\ne\lambda_{\ell+3}$ we must have $\lambda_{\ell},\lambda_{\ell+3}>0$ and thus $\nu_{\ell}=\nu_{\ell+3}=0$.\par
 \end{itemize}\par 
 Therefore, no opposite pair can be $(0,0)$ and no opposite pair can have a zero/positive split. Hence all $\lambda_j>0$, and by $\lambda_j\nu_j=0$ in \eqref{Eqns: Lagrange for n=6, (0,1,2,1,2,1)} we have all $\nu_j=0$.\par 
\medskip
\textbf{Step 2.} We show that for all $j\in\{0,1,2,3,4,5\}$ with $\lambda_j\neq \lambda_{j+3}$, we have $(\lambda_j-1)(\lambda_{j+3}-1)<0$ and $\lambda_j+\lambda_{j+3}>2$.\par 

Without loss of generality, suppose $\lambda_{j}<\lambda_{j+3}$. The function $F$ strictly increases on $(0,1)$ and strictly decreases on $(1,\infty)$, so  from \eqref{eq: opposite F} with $\nu_{j}=\nu_{j+3}=0$ proved in Step 1, we get
\begin{equation}
   F(\lambda_{j}) = F(\lambda_{j+3}).
\end{equation}
And hence
\[
  \lambda_{j}<1<\lambda_{j+3}.
\]
Define $\Theta(x)\coloneqq F(x)-F(2-x)$. Then $\Theta(1)=0$ and $\Theta'(x)=-\frac{2}{3}\log\big(x(2-x)\big)>0$ on $(0,1)$, so $\Theta(x)<0$ on $(0,1)$. Applying this to $x=\lambda_{j}$ gives $F(\lambda_{j})<F(2-\lambda_{j})$. Since $F$ decreases on $(1,\infty)$ and $F(\lambda_{j+3})=F(\lambda_{j})$, we deduce
\begin{equation}\label{Ineq: opposite sum lower bound}
  \lambda_{j+3}>2-\lambda_{j}
  \quad\Rightarrow\quad
  \lambda_{j}+\lambda_{j+3}>2.
\end{equation}\par
\medskip
\textbf{Step 3.} We show that for $\ell,\ell'\in\{0,1,2\}\setminus\{j_0\}$, we have $\lambda_{\ell}=\lambda_{\ell+3}=\lambda_{\ell'}=\lambda_{\ell'+3}$.\par 
By the assumption $\lambda_{j_0}\neq \lambda_{j_0+3}$, we have $\lambda_{j_0}+\lambda_{j_0+3}>2$ from Step 2. Hence
\[
  \lambda_{j_0}^2+\lambda_{j_0+3}^2 \ge \frac{(\lambda_{j_0}+\lambda_{j_0+3})^2}{2} > 2.
\]
Consequently, if $\lambda_{j}\neq \lambda_{j+3}$ for all $0\leq j\le 2$, then $\sum_{k=0}^{5}\lambda_k^2>6$, contradicting $\sum_{k=0}^{5}\lambda_k^2<6$ in \eqref{Eqns: Lagrange for n=6, (0,1,2,1,2,1)}. Therefore, there exists at least one $\ell$ with $\lambda_{\ell}=\lambda_{\ell+3}$. Let $\ell$ be an index in $\{0,1,2\}\setminus\{j_0\}$ such that $\lambda_\ell=\lambda_{\ell+3}$ and $\lambda_\ell$ is minimal among all pairs $\{\lambda_j=\lambda_{j+3}\}$ with $j\in\{0,1,2\}\setminus\{j_0\}$. If $\lambda_\ell=\lambda_{\ell+3}\ge 1$, then $\sum_{k=0}^{5}\lambda_k^2>6$, which again contradicts $\sum_{k=0}^{5}\lambda_k^2<6$. So there at least one $\ell$ with $\lambda_{\ell}=\lambda_{\ell+3}<1$. Denote by $\ell'$ the remaining index in $\{0,1,2\}\setminus\{j_0,\ell\}$. \par 

 Since $x\mapsto-3x\log x+4x$ is strictly increasing on $(0,1)$ and $\lambda_{\ell}<1$, from the first line of \eqref{Eqns: Lagrange for n=6, (0,1,2,1,2,1)} with index $\ell$ we get
\begin{equation}\label{Ineq: two opposites sum upper bound}
  \lambda_{j_0}+\lambda_{j_0+3}+\lambda_{\ell'}+\lambda_{\ell'+3}
  = -3\lambda_\ell\log(\lambda_\ell)+4\lambda_\ell<4.
\end{equation}
If $\lambda_{\ell'}\neq \lambda_{\ell'+3}$, then by \eqref{Ineq: opposite sum lower bound} also $\lambda_{\ell'}+\lambda_{\ell'+3}>2$, and together with  $\lambda_{j_0}+\lambda_{j_0+3}>2$ we get $\lambda_{j_0}+\lambda_{j_0+3}+\lambda_{\ell'}+\lambda_{\ell'+3}>4$, contradicting \eqref{Ineq: two opposites sum upper bound}. Hence $\lambda_{\ell'}=\lambda_{\ell'+3}<1$. We compare the equations for $j=\ell$ and $j=\ell'$ in the first line of \eqref{Eqns: Lagrange for n=6, (0,1,2,1,2,1)}:
\[
\begin{cases}
-\tfrac{2}{3}\lambda_{\ell}\log(\lambda_{\ell})+\tfrac{8}{9}\lambda_{\ell}
  -\tfrac{2}{9}(\lambda_{j_0}+2\lambda_{\ell'}+\lambda_{j_0+3})=0,\\
-\tfrac{2}{3}\lambda_{\ell'}\log(\lambda_{\ell'})+\tfrac{8}{9}\lambda_{\ell'}-\tfrac{2}{9}(2\lambda_{\ell}+\lambda_{j_0}+\lambda_{j_0+3})=0.
\end{cases}
\]
Subtracting and using that $x\mapsto -\tfrac{2}{3}x\log x+\tfrac{4}{3}x$ is strictly increasing on $(0,\sqrt{6})$ yields
\[
  \lambda_{\ell'}=\lambda_{\ell}.
\]
Therefore, together with Step 2, $\lambda_{\ell}=\lambda_{\ell+3}$ and $\lambda_{\ell'}=\lambda_{\ell'+3}$, we have 
\begin{equation}
  (\lambda_{j_0}-1)(\lambda_{j_0+3}-1)<0,\qquad \lambda_{\ell+3}=\lambda_{\ell}=\lambda_{\ell'}=\lambda_{\ell'+3}.
\end{equation}\par
\medskip
\textbf{Step 4.} We show that if $\lambda_{j_0}\neq \lambda_{j_0+3}$ and $\lambda_{\ell}=\lambda_{\ell+3}=\lambda_{\ell'}=\lambda_{\ell'+3}$, where the indices $\{j_0,\ell,\ell'\}$ are as in Step 3, then the system \eqref{Eqns: Lagrange for n=6, (0,1,2,1,2,1)} has no solution.\par 

Without loss of generality, assuming $\lambda_{j_0}<\lambda_{j_0+3}$ and $\lambda_{\ell}=\lambda_{\ell+3}=\lambda_{\ell'}=\lambda_{\ell'+3}$, the system \eqref{Eqns: Lagrange for n=6, (0,1,2,1,2,1)} reduces to
\begin{equation}
  \begin{cases}
\mathcal{L}_{j_0}\coloneqq-\tfrac{2}{3}\lambda_{j_0}\log(\lambda_{j_0})+\tfrac{7}{9}\lambda_{j_0}+\tfrac{1}{9}\lambda_{j_0+3}-\tfrac{8}{9}\lambda_{\ell}=0,\\
\mathcal{L}_{j_0+3}\coloneqq-\tfrac{2}{3}\lambda_{j_0+3}\log(\lambda_{j_0+3})+\tfrac{7}{9}\lambda_{j_0+3}+\tfrac{1}{9}\lambda_{j_0}-\tfrac{8}{9}\lambda_{\ell}=0,\\
\mathcal{L}_{\ell}\coloneqq-\tfrac{2}{3}\lambda_{\ell}\log(\lambda_{\ell})+\tfrac{4}{9}\lambda_{\ell}-\tfrac{2}{9}(\lambda_{j_0}+\lambda_{j_0+3})=0,\\
0<4\lambda_{\ell}^2+\lambda_{j_0}^2+\lambda_{j_0+3}^2<6,\\
 \lambda_j>0,\qquad j=j_0,j_0+3,\ell.
\end{cases}
\end{equation}

Let $\lambda_{j_0+3}=r\lambda_{j_0}$ with $r>1$. By solving the equation $\mathcal{L}_{j_0}=\mathcal{L}_{j_0+3}$ we get
\[
  \lambda_{j_0}=er^{-\frac{r}{r-1}},\qquad \lambda_{j_0+3}=r\lambda_{j_0}=er^{-\frac{1}{r-1}}.
\]
Moreover,
\[
  0=\mathcal{L}_{j_0}+\mathcal{L}_{j_0+3}+4\mathcal{L}_{\ell}
  =-\tfrac{2}{3}\lambda_{j_0}\log(\lambda_{j_0})-\tfrac{2}{3}\lambda_{j_0+3}\log(\lambda_{j_0+3})-\tfrac{8}{3}\lambda_{\ell}\log(\lambda_{\ell}),
\]
so
\[
  \lambda_{\ell}=-\frac{\lambda_{j_0}\log(\lambda_{j_0})+\lambda_{j_0+3}\log(\lambda_{j_0+3})}{4\log(\lambda_{\ell})}.
\]
On the other hand, $\mathcal{L}_{\ell}=0$ gives
\[
  \lambda_{\ell}=\frac{-\lambda_{j_0}-\lambda_{j_0+3}}{-2+3\log(\lambda_{\ell})}.
\]
Combining the two expressions for $\lambda_{\ell}$ yields
\begin{equation}\label{Eqn: First equation of log(lambda_0)}
  \log(\lambda_{\ell})
  =-\frac{2(\lambda_{j_0}\log(\lambda_{j_0})+\lambda_{j_0+3}\log(\lambda_{j_0+3}))}
          {4\lambda_{j_0}+4\lambda_{j_0+3}-3\lambda_{j_0}\log(\lambda_{j_0})-3\lambda_{j_0+3}\log(\lambda_{j_0+3})}
  =\frac{-2 r^2+4 r \log r+2}{r^2+6 r \log r-1}.
\end{equation}
Using $\mathcal{L}_{j_0}=0$ with $\lambda_{j_0+3}=r\lambda_{j_0}$, we also obtain
\[
  \lambda_{\ell}=\frac{9}{8}\left(-\frac{2}{3}\lambda_{j_0}\log(\lambda_{j_0})+\frac{7}{9}\lambda_{j_0}+\frac{1}{9}\lambda_{4}\right)
  =\frac{e r^{-\frac{r}{r-1}}\left(r^2+6 r \log r-1\right)}{8 (r-1)},
\]
hence
\begin{equation}\label{Eqn: Second equation of log(lambda_0)}
  \log(\lambda_{\ell})
  =\log \left(r^2+6 r \log r-1\right)-\log \big(8 (r-1)\big)-\frac{r \log r}{r-1}+1.
\end{equation}
Set
\[
  h(x)=\log(x^2+6 x \log x-1)-\log\big(8 (x-1)\big)-\frac{x \log x}{x-1}+1
        -\frac{-2 x^2+4 x \log x+2}{x^2+6 x \log x-1},\qquad x\in \mathbb{R}_+,
\]
where the values of $h$ and $h'$ at $x=1$ are understood by continuous extension $h(1)=h'(1)=0$. From \eqref{Eqn: First equation of log(lambda_0)} and \eqref{Eqn: Second equation of log(lambda_0)} we have $h(r)=0$. \par 

We will show that $h(r)>0$ by monotonicity of $h$. To this end, we introduce the auxiliary functions
\[
  \begin{aligned}
    h_1(x)&=(x-1)^2 \big(x^2+6 x \log x-1\big)^2 h'(x)\\
          &=36 x^2 \log ^3(x)-24 x \left(x^2-1\right) \log ^2(x)+\left(11 x^2+14 x+11\right) (x-1)^2 \log (x)-12 (x+1) (x-1)^3,\\
    h_2(x)&=xh_1'(x)\\
          &=-(x-1)^2 \left(37 x^2+10 x-11\right)+72 x^2 \log ^3(x)-12 x \left(6 x^2-9 x-2\right) \log ^2(x)\\
          &\quad+4 x \left(11 x^3-18 x^2-3 x+10\right) \log (x),\\
    h_3(x)&=xh_2''(x)\\
          &=8 \left(-17 x^3-15 x^2+6 \left(11 x^3-24 x^2+22 x+1\right) \log (x)+21 x+18 x \log ^3(x)-54 (x-2) x \log ^2(x)+11\right),\\
    h_4(x)&=xh_3'(x)=24 \left(5 x^3-58 x^2+2 \left(33 x^2-66 x+58\right) x \log (x)+51 x+6 x \log ^3(x)+18 (3-2 x) x \log ^2(x)+2\right),\\
    h_5(x)&=xh_4''(x)=48 \left(4 \left(45 x^2-73 x+28\right)+6 \left(33 x^2-40 x+12\right) \log (x)+(9-36 x) \log ^2(x)\right),\\
    h_6(x)&=xh_5'(x)=96 \left(279 x^2+3 \left(66 x^2-52 x+3\right) \log (x)-266 x-18 x \log ^2(x)+36\right),\\
    h_7(x)&=xh_6'(x)=96 \left(756 x^2-422 x-18 x \log ^2(x)+12 (33 x-16) x \log (x)+9\right),\\
    h_8(x)&=xh_7''(x)=1152 (225 x+(66 x-3) \log (x)-19).
  \end{aligned}
\]
Direct calculations yield the following values at $x=1$:
\[
\begin{aligned}
&h(1)=0,\quad h'(1)=0,\quad h_1'(1)=0,\quad h_2'(1)=0,\quad h_2''(1)=0,\quad h_3'(1)=0,\\
&h_4'(1)=0,\quad h_4''(1)=0,\quad h_5'(1)>0,\quad h_6'(1)>0,\quad h_7'(1)>0,\quad h_7''(1)>0,\quad h_8'(1)>0,
\end{aligned}
\]
and, for all $x\ge 1$,
\[
  h_8''(x)=\frac{3456 (22 x+1)}{x^2}>0.
\]
It then follows that $h'_8$ is strictly increasing on $[1, +\infty)$, so positive on the same interval. Consequently, by $h_8(1)=h_7''(1)>0$, $h_8$ is positive on $[1, +\infty)$.  We thus deduce that $h''_7$ is positive on $[1, +\infty)$, so is $h_7$. Repeating this reasoning backward, we finally prove that $h$ is strictly increasing on $[1, +\infty)$, thus $h(x)>h(1)=0$ for $x>1$, in particular, $h(r)>0$, which contradicts $h(r)=0$. This contradicts $\lambda_{j_0}<\lambda_{j_0+3}$, completing the proof.
\end{proof}

\begin{proof}[Proof of Theorem~\ref{Thm: Log Sobolev inequality n=6}]
By Lemma~\ref{Lem: pairwise-opposite-equality, n=6} and Lemma~\ref{Lem: one opposite pair unequal, n=6}, the system \eqref{Eqns: Lagrange for n=6, (0,1,2,1,2,1)} has no solution. Hence, by Lemma~\ref{Lem: KKT condition absorbing mu}, we conclude that $f_{\phi_6}\ge 0$ on $\mathbb{R}_+^6$. This completes the proof of Theorem~\ref{Thm: Log Sobolev inequality n=6}.
\end{proof}

\begin{proof}[Sketch proof of Theorem~\ref{Thm: Log Sobolev inequality n=4}]
Set
\[
  \Phi(4)=\frac{1}{4}F_4\diag(\phi_4)F_4^{-1}
  =\begin{pmatrix}[1.2]
    \frac{9}{40} & -\frac{1}{10} & -\frac{1}{40} & -\frac{1}{10} \\
    -\frac{1}{10} & \frac{9}{40} & -\frac{1}{10} & -\frac{1}{40} \\
    -\frac{1}{40} & -\frac{1}{10} & \frac{9}{40} & -\frac{1}{10} \\
    -\frac{1}{10} & -\frac{1}{40} & -\frac{1}{10} & \frac{9}{40}
  \end{pmatrix},\qquad f_{\phi_4}(\lambda)
  = 2\inner{\lambda,\Phi(4)\lambda}-\mathrm{H}_4[\lambda].
\]
By Lemma~\ref{Lem: KKT condition absorbing mu}, to show $f_{\phi_4}\ge 0$ on $\mathbb{R}_+^4$, it suffices to show that the following system (indices modulo $4$) has no solution:
\begin{equation}\label{Eqns: Lagrange for n=4, (0,1,8/5,1)}
\begin{cases}
-\lambda_j\log(\lambda_j)
  +\Big(\tfrac{9}{10}\lambda_j-\tfrac{2}{5}(\lambda_{j+1}+\lambda_{j-1})-\tfrac{1}{10}\lambda_{j+2}\Big)
  -\nu_j=0,\qquad 0\le j\le 3,\\
0<\sum_{k=0}^3 \lambda_k^2<4,\\
\lambda_j\geq 0,\qquad 0\le j\le 3,\\
\lambda_j\nu_j=0,\qquad 0\le j\le 3,\\
\nu_j\ge 0,\qquad 0\le j\le 3.
\end{cases}
\end{equation}

As in the case $n=6$, we analyze \eqref{Eqns: Lagrange for n=4, (0,1,8/5,1)} in two complementary regimes:
\[
  \text{(i) } \lambda_j=\lambda_{j+2}\ \text{ for } j=0,1,
  \qquad
  \text{(ii) } \lambda_{j_0}\ne \lambda_{j_0+2}\ \text{ for some } j_0\in\{0,1\}.
\]
\begin{lemma}\label{Lem: pairwise-opposite-equality, n=4}
  The system \eqref{Eqns: Lagrange for n=4, (0,1,8/5,1)} has no solution in the region 
\[
  \{(\lambda,\nu)\in R^{4}_+\times R^{4}_+: \forall j\in \{0,1\},\lambda_{j}=\lambda_{j+2}\}.
\]
\end{lemma}
\begin{lemma}\label{Lem: one opposite pair unequal, n=4}
The system \eqref{Eqns: Lagrange for n=4, (0,1,8/5,1)} has no solution in the region 
\[
  \{(\lambda,\nu)\in R^{4}_+\times R^{4}_+: \exists j_0\in \{0,1\},\lambda_{j_0}\ne \lambda_{j_0+2}\}.
\]
\end{lemma}

The proofs of these lemmas follow the same stationary-condition arguments used for $n=6$, so we only give a sketch of proof: we first shows verbatim that any solution of \eqref{Eqns: Lagrange for n=4, (0,1,8/5,1)} satisfies $\lambda_j>0$ and $\nu_j=0$ for all $j$. In case (i), the symmetry $\lambda_j=\lambda_{j+2}$ reduces \eqref{Eqns: Lagrange for n=4, (0,1,8/5,1)} to
\[
\begin{cases}
-\lambda_j\log(\lambda_j)+\tfrac{8}{5}\lambda_j-\tfrac{4}{5}(\lambda_j+\lambda_{j+1})=0,\quad j=0,1,\\
0<\lambda_0^2+\lambda_1^2<2,\\
\lambda_0,\lambda_1>0,
\end{cases}
\]
and the strict monotonicity of $x\mapsto -x\log x+\tfrac{8}{5}x$ on $(0,\sqrt{2})$ forces $\lambda_0=\lambda_1<1$ as in the proof of Lemma \ref{Lem: pairwise-opposite-equality, n=6}, which contradicts the first equality. In case (ii), by the same analysis used in Lemma \ref{Lem: one opposite pair unequal, n=6}, we have $\lambda_{j_0}^2+\lambda_{j_0+2}^2>2$ and $\lambda_{\ell}=\lambda_{\ell+2}<1$. The system becomes
\[
\begin{cases}
-\lambda_{j_0}\log(\lambda_{j_0})+\tfrac{9}{10}\lambda_{j_0}-\tfrac{1}{10}\lambda_{j_0+2}-\tfrac{4}{5}\lambda_{\ell}=0,\\
-\lambda_{j_0+2}\log(\lambda_{j_0+2})+\tfrac{9}{10}\lambda_{j_0+2}-\tfrac{1}{10}\lambda_{j_0}-\tfrac{4}{5}\lambda_{\ell}=0,\\
-\lambda_{\ell}\log(\lambda_{\ell})+\tfrac{4}{5}\lambda_{\ell}-\tfrac{2}{5}(\lambda_{j_0}+\lambda_{j_0+2})=0,\\
0<2\lambda_{\ell}^2+\lambda_{j_0}^2+\lambda_{j_0+2}^2<4,\\
\lambda_{j_0},\lambda_{j_0+2},\lambda_{\ell}>0.
\end{cases}
\]
Write $\lambda_{j_0+2}=r\lambda_{j_0}$ with $r>1$. From the above system, eliminating $\log(\lambda_\ell)$ like Step 3 in the proof of Lemma \ref{Lem: one opposite pair unequal, n=6} yields the scalar identity $h(r)=0$ with
\[
  h(x)=\log(4 x \log x-\tfrac{2}{5} (x^2-1))+\frac{-4 x^2+8 x \log x+4}{x^2-10 x \log x-1}
       -\log(\tfrac{16 (x-1)}{5})-\frac{x \log x}{x-1}+1,\qquad x\in \mathbb{R}_+,
\]
where the values of $h$ and $h'$ at $x=1$ are understood by continuous extension $h(1)=h'(1)=0$. As in Step 4 in the proof of Lemma~\ref{Lem: one opposite pair unequal, n=6}, we introduce the auxiliary functions
\[
\begin{aligned}
& h_1(x)=(x-1)^2 (4 x \log (x)-\tfrac{2}{5}(x^2-1))^2h'(x),\quad h_2(x)=xh_1'(x),\quad h_3(x)=xh_2'(x),\quad h_3(x)=xh_2''(x),\\
&h_4(x)=xh_3'(x),\quad h_5(x)=xh_4''(x),\quad h_6(x)=xh_5'(x),\quad h_7(x)=x^2h_6''(x),\quad h_8(x)=xh_7''(x),
\end{aligned}
\]
and $h_8$ is positive on $[1, +\infty)$, so inductively we can derive $h'(x)>0$ for all $x>1$, hence $h(x)>0$ for $x>1$, a contradiction. Therefore, \eqref{Eqns: Lagrange for n=4, (0,1,8/5,1)} has no solution in either regime. By Lemma~\ref{Lem: KKT condition absorbing mu}, we conclude that $f_{\phi_4}\ge 0$ on $\mathbb{R}_+^4$, completing the proof of Theorem~\ref{Thm: Log Sobolev inequality n=4}.
\end{proof}

\section{Dyadic induction to optimal LSI on $\mathbb{Z}_{6\cdot 2^k}$ and $\mathbb{Z}_{8\cdot 2^k}$}\label{Sec: Induction from n to 2n}
To pass from the $n$--LSI to the $2n$--LSI, we compare the Dirichlet forms associated with the symbols $\gamma_n$ and $\gamma_{2n}$ via the Cooley--Tukey factorization of the $2n$-point DFT into two $n$-point DFTs \cite{MR178586}. We start with the following proposition, which provides the key comparison between the $n$-- and $2n$--level Dirichlet forms under appropriate compatibility assumptions on $\gamma_n$ and $\gamma_{2n}$.\par

\begin{proposition}\label{Prop: Compare Dirichlet form for n and 2n}
Let $n\ge 3$ be an integer. Let $\gamma_n:\mathbb{Z}_n\to\mathbb{R}_+$ be a function with associated matrix
\[
  \Gamma(n)=\frac{1}{n}F_n\diag(\gamma_n)F_n^{-1}\in M_n(\mathbb{C}).
\]
For $\lambda=(a_0,b_0,\ldots,a_{n-1},b_{n-1})\in\mathbb{R}_+^{2n}$ set $a=(a_0,\ldots,a_{n-1})\in\mathbb{R}_+^n$ and $b=(b_0,\ldots,b_{n-1})\in\mathbb{R}_+^n$. Assume $\gamma_n$ and $\gamma_{2n}$ satisfy
\begin{equation}\label{Cond: gamma pair condition}
\begin{cases}
\gamma_{n}(0)=\gamma_{2n}(0)=0,\\
\gamma_{n}(k)=\gamma_{n}(n-k),\qquad 1\le k\le n-1,\\
\gamma_{2n}(k)=\gamma_{2n}(2n-k),\qquad 1\le k\le 2n-1,\\
\gamma_{2n}(k)\ge \gamma_n(k),\qquad 1\le k\le [\frac{n-1}{2}],\\
\gamma_{2n}(n-k)-\gamma_{2n}(k)-1\ge 0,\qquad 0\le k\le [\frac{n-1}{2}].
\end{cases}
\end{equation}
Then the following inequality holds:
\begin{equation}\label{Ineq: Compare Dirichlet form for n and 2n}
  \inner{a,\Gamma(n)a}+\inner{b,\Gamma(n)b}+\frac{1}{2n}\big(\norm{a}_2-\norm{b}_2\big)^2 \le 2\inner{\lambda,\Gamma(2n)\lambda},
\end{equation}
provided the following inequality holds for all $x\ge 0$ and $0\le r_a,r_b\le 1$:
\begin{equation}\label{Ineq: Quadratic inequality from n to 2n}
 \begin{cases}
  \bigl(2\gamma_{2n}(\tfrac{n}{2})-2\gamma_n(\tfrac{n}{2})-1\bigr)\bigl(r_a+r_b x^2\bigr)
  +\gamma_{2n}(n)(1-x)^2-(1+x^2)+2x\sqrt{(1+r_a)(1+r_b)} \ge 0,& \text{if $n$ is even},\\
  \gamma_{2n}(n)(1-x)^2-(1+x^2)+2x \ge 0,& \text{if $n$ is odd}.
 \end{cases}
\end{equation}
\end{proposition}

\begin{proof}
  We first prove the theorem for even $n$. The condition \eqref{Cond: gamma pair condition} becomes 
  \begin{equation}\label{Cond: gamma pair condition, n even}
    \begin{cases}
\gamma_{n}(0)=\gamma_{2n}(0)=0,\\
\gamma_{n}(k)=\gamma_{n}(n-k),\qquad 1\le k\le n-1,\\
\gamma_{2n}(k)=\gamma_{2n}(2n-k),\qquad 1\le k\le 2n-1,\\
\gamma_{2n}(k)\ge \gamma_n(k),\qquad 1\le k\le \frac{n}{2}-1,\\
\gamma_{2n}(n-k)-\gamma_{2n}(k)-1\ge 0,\qquad 0\le k\le \frac{n}{2}-1.
\end{cases}
  \end{equation}
Write $\lambda=(a_0,b_0,a_1,b_1,\ldots,a_{n-1},b_{n-1})$. The DFTs $\hat{\lambda}=F_{2n}\lambda$, $\hat{a}=F_{n}a=(\hat{a}_0,\ldots,\hat{a}_{n-1})$ and $\hat{b}=F_{n}b=(\hat{b}_0,\ldots,\hat{b}_{n-1})$ satisfy
\[
    \hat{\lambda}_k=\hat{a}_k+e^{-\frac{2\pi i k}{2n}}\hat{b}_k,\qquad \hat{\lambda}_{n+k}=\hat{a}_k-e^{-\frac{2\pi i k}{2n}}\hat{b}_k,
\]
for $ k=0,\ldots,n-1$. Set
\[
  D=\diag\big(1,e^{-\frac{2\pi i}{2n}},e^{-\frac{4\pi i}{2n}},\ldots,e^{-\frac{2(n-1)\pi i}{2n}}\big),
\]
so that
\[
  \hat{\lambda}=\begin{pmatrix}\hat{a}+D\hat{b}\\ \hat{a}-D\hat{b}\end{pmatrix}.
\]
 Using that $\big(\tfrac{1}{\sqrt{m}}F_m\big)^2$ is a permutation matrix which reverses indices $k\leftrightarrow m-k$ (with $0$ fixed) for $1\le k\le m-1$, we have
\[
  \frac{1}{2n}F_{2n}\Gamma(2n)F_{2n}^{-1}=\frac{1}{(2n)^2}F_{2n}^2\diag(\gamma_{2n})F_{2n}^{-2}
  =\diag\Big(0,\tfrac{\gamma_{2n}(1)}{(2n)^2},\ldots,\tfrac{\gamma_{2n}(2n-1)}{(2n)^2}\Big)
  =\begin{pmatrix} M_1 & 0 \\ 0 & M_2 \end{pmatrix},
\]
where $M_1,M_2$ are $n\times n$ diagonal matrices for the frequency blocks $\{0,1,\ldots,n-1\}$ and $\{n,\ldots,2n-1\}$, respectively. Therefore
\[
\begin{aligned}
  2\inner{\lambda,\Gamma(2n)\lambda} &=2\inner{\tfrac{1}{\sqrt{2n}}F_{2n}\lambda,\tfrac{1}{\sqrt{2n}}F_{2n}\Gamma(2n)\lambda}\\
  &=2\inner{\begin{pmatrix}\hat{a}+D\hat{b}\\ \hat{a}-D\hat{b}\end{pmatrix},
        \begin{pmatrix} M_1 & 0 \\ 0 & M_2 \end{pmatrix}
        \begin{pmatrix}\hat{a}+D\hat{b}\\ \hat{a}-D\hat{b}\end{pmatrix}}\\
  &=\inner{\hat{a},4M_1\hat{a}}+\inner{\hat{b},4M_1\hat{b}}
    +\inner{\hat{a}-D\hat{b},2(M_2-M_1)(\hat{a}-D\hat{b})}.
\end{aligned}
\]
Here
\[
  \begin{aligned}
    4M_1&=\diag\Big(0,\tfrac{\gamma_{2n}(1)}{n^2},\ldots,\tfrac{\gamma_{2n}(n-1)}{n^2}\Big),\\
    2(M_2-M_1)&=\diag\Big(\tfrac{\gamma_{2n}(n)}{2n^2},\tfrac{\gamma_{2n}(n+1)-\gamma_{2n}(1)}{2n^2},\ldots,\tfrac{\gamma_{2n}(2n-1)-\gamma_{2n}(n-1)}{2n^2}\Big).
  \end{aligned}
\]
Recall that we write $\hat{a}=(\hat{a}_0,\ldots,\hat{a}_{n-1})$ and $\hat{b}=(\hat{b}_0,\ldots,\hat{b}_{n-1})$. For real vectors $a,b$ we have the conjugate symmetries
\[
  \hat a_k=\overline{\hat a_{n-k}},
  \qquad
  (D\hat b)_k=e^{-\frac{2\pi i k}{2n}}\hat b_k
  =-\overline{(D\hat b)_{n-k}},
\]
so
\[
  \inner{\hat a,4M_1\hat a}+\inner{\hat b,4M_1\hat b}
  =\sum_{k=1}^{\frac{n}{2}-1}\frac{\gamma_{2n}(k)+\gamma_{2n}(n-k)}{n^2}\big(\abs{\hat a_k}^2+\abs{\hat b_k}^2\big)
   +\frac{\gamma_{2n}(\frac{n}{2})}{n^2}\big(\abs{\hat a_{\frac{n}{2}}}^2+\abs{\hat b_{\frac{n}{2}}}^2\big).
\]
Using $\gamma_{2n}(k)=\gamma_{2n}(2n-k)$ and pairing $k$ with $n-k$,
\[
\begin{aligned}
\inner{\hat a-D\hat b,2(M_2-M_1)(\hat a-D\hat b)}
&=\frac{\gamma_{2n}(n)}{2n^2}\abs{\hat a_0-\hat b_0}^2
  +\sum_{k=1}^{\frac{n}{2}-1}\frac{\gamma_{2n}(n+k)-\gamma_{2n}(k)}{2n^2}\abs{\hat a_k-(D\hat b)_k}^2\\
&\quad+\sum_{k=1}^{\frac{n}{2}-1}\frac{\gamma_{2n}(2n-k)-\gamma_{2n}(n-k)}{2n^2}\abs{\hat a_{n-k}-(D\hat b)_{n-k}}^2\\
&=\frac{\gamma_{2n}(n)}{2n^2}\abs{\hat a_0-\hat b_0}^2
  +\sum_{k=1}^{\frac{n}{2}-1}\frac{\gamma_{2n}(n-k)-\gamma_{2n}(k)}{2n^2}
     \Big(\abs{\hat a_k-(D\hat b)_k}^2-\abs{\hat a_k+(D\hat b)_k}^2\Big)\\
&=\frac{\gamma_{2n}(n)}{2n^2}\abs{\hat a_0-\hat b_0}^2
  -\sum_{k=1}^{\frac{n}{2}-1}\frac{2\big(\gamma_{2n}(n-k)-\gamma_{2n}(k)\big)}{n^2}
    \Re\big(\overline{\hat a_k}(D\hat b)_k\big),
\end{aligned}
\]
where we have used $\abs{\hat a_{n-k}-(D\hat b)_{n-k}}^2=\abs{\hat a_k+(D\hat b)_k}^2$ and the identity $\abs{x-y}^2-\abs{x+y}^2=-4\Re(\overline{x}y)$.\par

Hence, on the frequency side,
\begin{equation}\label{Eqn: 2n quadratic form of lambda}
\begin{aligned}
2\inner{\lambda,\Gamma(2n)\lambda}
&=\sum_{k=1}^{\frac{n}{2}-1}\frac{\gamma_{2n}(k)+\gamma_{2n}(n-k)}{n^2}\big(\abs{\hat a_k}^2+\abs{\hat b_k}^2\big)
  +\frac{\gamma_{2n}(\tfrac{n}{2})}{n^2}\big(\abs{\hat a_{\frac{n}{2}}}^2+\abs{\hat b_{\frac{n}{2}}}^2\big)\\
&\quad+\frac{\gamma_{2n}(n)}{2n^2}\abs{\hat a_0-\hat b_0}^2
  -\sum_{k=1}^{\frac{n}{2}-1}\frac{2\big(\gamma_{2n}(n-k)-\gamma_{2n}(k)\big)}{n^2}
    \Re\big(\overline{\hat a_k}(D\hat b)_k\big).
\end{aligned}
\end{equation}
Similarly, using $\gamma_{n}(k)=\gamma_{n}(n-k)$,
\begin{equation}\label{Eqn: n quadratic form of a and b}
\begin{aligned}
\inner{a,\Gamma(n)a}+\inner{b,\Gamma(n)b}
&=\sum_{k=1}^{n-1}\frac{\gamma_{n}(k)}{n^2}\big(\abs{\hat a_k}^2+\abs{\hat b_k}^2\big)\\
&=\sum_{k=1}^{\frac{n}{2}-1}\frac{\gamma_{n}(k)+\gamma_{n}(n-k)}{n^2}\big(\abs{\hat a_k}^2+\abs{\hat b_k}^2\big)
  +\frac{\gamma_{n}(\tfrac{n}{2})}{n^2}\big(\abs{\hat a_{\frac{n}{2}}}^2+\abs{\hat b_{\frac{n}{2}}}^2\big)\\
&=\sum_{k=1}^{\frac{n}{2}-1}\frac{2\gamma_{n}(k)}{n^2}\big(\abs{\hat a_k}^2+\abs{\hat b_k}^2\big)
  +\frac{\gamma_{n}(\tfrac{n}{2})}{n^2}\big(\abs{\hat a_{\frac{n}{2}}}^2+\abs{\hat b_{\frac{n}{2}}}^2\big).
\end{aligned}
\end{equation}
Therefore, combining \eqref{Eqn: 2n quadratic form of lambda} and \eqref{Eqn: n quadratic form of a and b}, \eqref{Ineq: Compare Dirichlet form for n and 2n} is equivalent to
\[
\begin{aligned}
&\sum_{k=1}^{\frac{n}{2}-1}\frac{2\gamma_{n}(k)}{n^2}\big(\abs{\hat a_k}^2+\abs{\hat b_k}^2\big)
 +\frac{\gamma_{n}(\frac{n}{2})}{n^2}\big(\abs{\hat a_{\frac{n}{2}}}^2+\abs{\hat b_{\frac{n}{2}}}^2\big)
 +\frac{1}{2n}\big(\norm{a}_2-\norm{b}_2\big)^2\\
&\qquad\le\sum_{k=1}^{\frac{n}{2}-1}\frac{\gamma_{2n}(k)+\gamma_{2n}(n-k)}{n^2}\big(\abs{\hat a_k}^2+\abs{\hat b_k}^2\big)-\sum_{k=1}^{\frac{n}{2}-1}\frac{2\big(\gamma_{2n}(n-k)-\gamma_{2n}(k)\big)}{n^2}\Re\big(\overline{\hat a_k}(D\hat b)_k\big)\\
&\qquad\quad +\frac{\gamma_{2n}(\frac{n}{2})}{n^2}\big(\abs{\hat a_{\frac{n}{2}}}^2+\abs{\hat b_{\frac{n}{2}}}^2\big)
 +\frac{\gamma_{2n}(n)}{2n^2}\abs{\hat a_0-\hat b_0}^2.
\end{aligned}
\]
Using $\abs{x-y}^2=\abs{x}^2+\abs{y}^2-2\Re(\overline{x}y)$, we obtain
\[
\begin{aligned}
&\sum_{k=1}^{\frac{n}{2}-1}\frac{\gamma_{2n}(k)+\gamma_{2n}(n-k)-2\gamma_n(k)}{n^2}\big(\abs{\hat a_k}^2+\abs{\hat b_k}^2\big)
-\sum_{k=1}^{\frac{n}{2}-1}\frac{2\big(\gamma_{2n}(n-k)-\gamma_{2n}(k)\big)}{n^2}\Re\big(\overline{\hat a_k}(D\hat b)_k\big)\\
&\qquad=\sum_{k=1}^{\frac{n}{2}-1}\frac{2\big(\gamma_{2n}(k)-\gamma_n(k)\big)}{n^2}\big(\abs{\hat a_k}^2+\abs{\hat b_k}^2\big)
+\sum_{k=1}^{\frac{n}{2}-1}\frac{\gamma_{2n}(n-k)-\gamma_{2n}(k)}{n^2}\abs{\hat a_k-(D\hat b)_k}^2\\
&\qquad\ge \sum_{k=1}^{\frac{n}{2}-1}\frac{2\big(\gamma_{2n}(k)-\gamma_n(k)\big)}{n^2}\big(\abs{\hat a_k}^2+\abs{\hat b_k}^2\big)
+\sum_{k=1}^{\frac{n}{2}-1}\frac{\gamma_{2n}(n-k)-\gamma_{2n}(k)}{n^2}\big(\abs{\hat a_k}-\abs{\hat b_k}\big)^2,
\end{aligned}
\]
where the last step uses the assumption $\gamma_{2n}(n-k)-\gamma_{2n}(k)\geq 1$ in \eqref{Cond: gamma pair condition}. Therefore, to prove \eqref{Ineq: Compare Dirichlet form for n and 2n}, it suffices to show
\begin{equation}\label{Ineq: Fourier side of comparison of Dirichlet form for n and 2n}
\begin{aligned}
\frac{1}{2n}\big(\norm{a}_2-\norm{b}_2\big)^2&\leq\sum_{k=1}^{\frac{n}{2}-1}\frac{2\big(\gamma_{2n}(k)-\gamma_n(k)\big)}{n^2}\big(\abs{\hat a_k}^2+\abs{\hat b_k}^2\big)
+\sum_{k=1}^{\frac{n}{2}-1}\frac{\gamma_{2n}(n-k)-\gamma_{2n}(k)}{n^2}\big(\abs{\hat a_k}-\abs{\hat b_k}\big)^2\\
&\quad
+\frac{\gamma_{2n}(\frac{n}{2})-\gamma_n(\frac{n}{2})}{n^2}\big(\abs{\hat a_{\frac{n}{2}}}^2+\abs{\hat b_{\frac{n}{2}}}^2\big)
+\frac{\gamma_{2n}(n)}{2n^2}\abs{\hat a_0-\hat b_0}^2.
\end{aligned}
\end{equation}\par 

Now estimate $\big(\norm{a}_2-\norm{b}_2\big)^2$. Let
\[
  \begin{aligned}
    u&=\Big(\sqrt{\abs{\hat a_0}^2+\abs{\hat a_{\frac{n}{2}}}^2}, \sqrt{2}\abs{\hat a_1},\sqrt{2}\abs{\hat a_2},\ldots,\sqrt{2}\abs{\hat a_{\frac{n}{2}-1}}\Big),\\
    v&=\Big(\sqrt{\abs{\hat b_0}^2+\abs{\hat b_{\frac{n}{2}}}^2}, \sqrt{2}\abs{\hat b_1},\sqrt{2}\abs{\hat b_2},\ldots,\sqrt{2}\abs{\hat b_{\frac{n}{2}-1}}\Big),
  \end{aligned}
\]
so that $\norm{a}_2=\frac{1}{\sqrt{n}}\norm{u}_2$ and $\norm{b}_2=\frac{1}{\sqrt{n}}\norm{v}_2$. Then
\[
\begin{aligned}
\big(\norm{a}_2-\norm{b}_2\big)^2
&=\frac{1}{n}\big(\norm{u}_2-\norm{v}_2\big)^2\\
&\le \frac{1}{n}\norm{u-v}_2^2\\
&=\frac{1}{n}\left(
\Big(\sqrt{\abs{\hat a_0}^2+\abs{\hat a_{\frac{n}{2}}}^2}-\sqrt{\abs{\hat b_0}^2+\abs{\hat b_{\frac{n}{2}}}^2}\Big)^2
+2\sum_{k=1}^{\frac{n}{2}-1}\big(\abs{\hat a_k}-\abs{\hat b_k}\big)^2\right)\\
&=\frac{1}{n}\left(
\abs{\hat a_0}^2+\abs{\hat a_{\frac{n}{2}}}^2+\abs{\hat b_0}^2+\abs{\hat b_{\frac{n}{2}}}^2
-2\sqrt{\big(\abs{\hat a_0}^2+\abs{\hat a_{\frac{n}{2}}}^2\big)\big(\abs{\hat b_0}^2+\abs{\hat b_{\frac{n}{2}}}^2\big)}
+2\sum_{k=1}^{\frac{n}{2}-1}\big(\abs{\hat a_k}-\abs{\hat b_k}\big)^2\right).
\end{aligned}
\]
Plugging this into \eqref{Ineq: Fourier side of comparison of Dirichlet form for n and 2n}, we see it is enough to prove
\begin{equation}
\begin{aligned}
&\sum_{k=1}^{\frac{n}{2}-1}\frac{2\big(\gamma_{2n}(k)-\gamma_n(k)\big)}{n^2}\big(\abs{\hat a_k}^2+\abs{\hat b_k}^2\big)
+\sum_{k=1}^{\frac{n}{2}-1}\frac{\gamma_{2n}(n-k)-\gamma_{2n}(k)-1}{n^2}\big(\abs{\hat a_k}-\abs{\hat b_k}\big)^2
-\frac{1}{2n^2}\big(\abs{\hat a_0}^2+\abs{\hat b_0}^2\big)\\
&\qquad
+\frac{2\gamma_{2n}(\frac{n}{2})-2\gamma_n(\frac{n}{2})-1}{2n^2}\big(\abs{\hat a_{\frac{n}{2}}}^2+\abs{\hat b_{\frac{n}{2}}}^2\big)
+\frac{\gamma_{2n}(n)}{2n^2}\abs{\hat a_0-\hat b_0}^2
+\frac{1}{n^2}\sqrt{\big(\abs{\hat a_0}^2+\abs{\hat a_{\frac{n}{2}}}^2\big)\big(\abs{\hat b_0}^2+\abs{\hat b_{\frac{n}{2}}}^2\big)}
\ \ge\ 0.
\end{aligned}
\end{equation}
Under the pairwise conditions \eqref{Cond: gamma pair condition, n even}, the first two sums are nonnegative. Thus it suffices to verify
\begin{equation}\label{Ineq: Sufficient inequality of comparison of Dirichlet form for n and 2n}
\begin{aligned}
&-\big(\abs{\hat a_0}^2+\abs{\hat b_0}^2\big)+(2\gamma_{2n}(\tfrac{n}{2})-2\gamma_n(\tfrac{n}{2})-1)\big(\abs{\hat a_{\frac{n}{2}}}^2+\abs{\hat b_{\frac{n}{2}}}^2\big)
+\gamma_{2n}(n)\abs{\hat a_0-\hat b_0}^2\\
&\qquad+2\sqrt{\big(\abs{\hat a_0}^2+\abs{\hat a_{\frac{n}{2}}}^2\big)\big(\abs{\hat b_0}^2+\abs{\hat b_{\frac{n}{2}}}^2\big)}\ \ge\ 0.
\end{aligned}
\end{equation}
Note that by definition, $\hat{a}_{\frac{n}{2}}=\sum_{j=0}^{n-1}(-1)^{j}a_j$, $\hat{a}_{0}=\sum_{j=0}^{n-1}a_j$ and similarly for $\hat{b}_{\frac{n}{2}}$ and $\hat{b}_{0}$. Since $a,b\in\mathbb{R}_+^n$, we have
\[
  \abs{\hat a_{\frac{n}{2}}}\le \hat a_0,\qquad \abs{\hat b_{\frac{n}{2}}}\le \hat b_0.
\]
Write $\abs{\hat a_{\frac{n}{2}}}=\sqrt{r_a}\hat a_0$ and $\abs{\hat b_{\frac{n}{2}}}=\sqrt{r_b}\hat b_0$ with $0\le r_a,r_b\le 1$. By symmetry in $\hat a_0$ and $\hat b_0$ in \eqref{Ineq: Sufficient inequality of comparison of Dirichlet form for n and 2n}, assume $\hat a_0>0$ and set $x=\hat b_0/\hat a_0\ge 0$. Then \eqref{Ineq: Sufficient inequality of comparison of Dirichlet form for n and 2n} becomes
\[
  \bigl(2\gamma_{2n}(\tfrac{n}{2})-2\gamma_n(\tfrac{n}{2})-1\bigr)\bigl(r_a+r_b x^2\bigr)
  +\gamma_{2n}(n)(1-x)^2-(1+x^2)+2x\sqrt{(1+r_a)(1+r_b)} \ge 0,
\]
which is exactly \eqref{Ineq: Quadratic inequality from n to 2n} if $n$ is even.\par 

The case where $n$ is odd is handled by exactly the same computation. In this case the condition \eqref{Cond: gamma pair condition} becomes 
\[
  \begin{cases}
    \gamma_n(0)=0,\\
    \gamma_n(k)=\gamma_n(n-k),\qquad 1\le k\le n-1,\\
    \gamma_{2n}(k)=\gamma_{2n}(2n-k),\qquad 1\le k\le 2n-1,\\
    \gamma_{2n}(k)\ge \gamma_n(k),\qquad 1\le k\le \tfrac{n-1}{2},\\
    \gamma_{2n}(n-k)-\gamma_{2n}(k)-1\ge 0,\qquad 0\le k\le \tfrac{n-1}{2}.
  \end{cases}
\]
Since for odd $n$ the middle frequencies $\hat a_{\frac{n}{2}}$ and $\hat b_{\frac{n}{2}}$ are absent in \eqref{Ineq: Sufficient inequality of comparison of Dirichlet form for n and 2n}, setting $x=\hat b_0/\hat a_0\ge 0$ and $\hat{a}_{\frac{n}{2}}=\hat{b}_{\frac{n}{2}}=0$, the condition \eqref{Ineq: Sufficient inequality of comparison of Dirichlet form for n and 2n} reduces to the scalar condition
\[
  \gamma_{2n}(n)(1-x)^2-(1+x^2)+2x\ge 0,
\]
which is \eqref{Ineq: Quadratic inequality from n to 2n} if $n$ is odd. This completes the proof of the theorem.
\end{proof}

Under the conditions \eqref{Cond: gamma pair condition} and \eqref{Ineq: Quadratic inequality from n to 2n}, we now pass from the $n$--LSI to the $2n$--LSI by decomposing the entropy functional $\mathrm{H}_{2n}$ for 
$\lambda\in\mathbb{R}^{2n}_+$, applying the $n$--LSI and the $2$--LSI to the resulting components, and then using Proposition~\ref{Prop: Compare Dirichlet form for n and 2n} to compare the resulting Dirichlet forms 
with the $2n$-level form.

\begin{theorem}\label{Thm: induction from n-LSI to 2n-LSI}
Let $n\geq 3$ be an integer. Let $\Gamma(n)$ and $\Gamma(2n)$ be the matrices corresponding to a length-function pair $(\gamma_n,\gamma_{2n})$ satisfying \eqref{Cond: gamma pair condition} and \eqref{Ineq: Quadratic inequality from n to 2n}. If the $n$--LSI holds with Dirichlet form $\inner{\lambda,\Gamma(n)\lambda}$ and constant $2$:
\[
\mathrm{H}_n[\lambda]\le 2\inner{\lambda,\Gamma(n)\lambda}\qquad \forall\lambda\in \mathbb{R}_+^n,
\]
then the $2n$--LSI holds with Dirichlet form $\inner{\lambda,\Gamma(2n)\lambda}$ and the same constant $2$:
\[
\mathrm{H}_{2n}[\lambda]\le 2\inner{\lambda,\Gamma(2n)\lambda}\qquad \forall\lambda\in \mathbb{R}_+^{2n}.
\]\par 

In addition, for odd $n_0\ge 3$, if the $n_0$--LSI holds for the word-length function $\psi_{n_0}$ with constant $2$, then for every $m\ge 1$ the $n_0\cdot 2^m$--LSI holds for $\psi_{n_0\cdot 2^m}$ with constant $2$.
\end{theorem}

The last statement in the above theorem does not directly follow for even $n_0$. Indeed, the pair $(\psi_{n_0},\psi_{2n_0})$ does not satisfy \eqref{Ineq: Quadratic inequality from n to 2n} in Proposition~\ref{Prop: Compare Dirichlet form for n and 2n}. To overcome this, we will later introduce pairs of modified weight functions $(\gamma_n,\gamma_{2n})$ in \eqref{Eqn: definition of induced gamma_n for odd base case} and \eqref{Eqn: defination of induced gamma_n} that do satisfy \eqref{Ineq: Quadratic inequality from n to 2n}, thereby enabling the induction that derives Theorem~\ref{Thm: Log Sobolev inequality n=6 times 2^k and n=8 times 2^k}.

\begin{proof}
Recall the case $n=2$ for Theorem~\ref{Thm: Log Sobolev inequality n=6 times 2^k and n=8 times 2^k} is known in \cite{MR420249}, i.e.,
\begin{equation}\label{Ineq: 2-LSI}
  \frac{1}{4}\left( x^2\log\frac{2x^2}{x^2+y^2}+y^2\log\frac{2y^2}{x^2+y^2}\right)
\le \left( \frac{x-y}{2} \right)^2\qquad x,y\geq 0.
\end{equation}
Given $\lambda=(a_0,b_0,a_1,b_1,\ldots,a_{n-1},b_{n-1})$ with $a=(a_0,\ldots,a_{n-1})$ and $b=(b_0,\ldots,b_{n-1})$, the $2n$--entropy splits as
\[
\begin{aligned}
\mathrm{H}_{2n}[\lambda]&=\frac{1}{2n}\sum_{i=0}^{n-1} a_i^2 \log(\frac{2na_i^2}{\sum_{i=0}^{n-1}(a_i^2+b_i^2)})+\frac{1}{2n}\sum_{i=0}^{n-1} b_i^2 \log(\frac{2nb_i^2}{\sum_{i=0}^{n-1}(a_i^2+b_i^2)})\\
&=\frac{1}{2n}\sum_{i=0}^{n-1} a_i^2\log(\frac{n a_i^2}{\sum_{i=0}^{n-1}a_i^2})+\frac{1}{2n}\sum_{i=0}^{n-1} b_i^2\log(\frac{n b_i^2}{\sum_{i=0}^{n-1}b_i^2})\\
 &\quad +\frac{1}{2n}\left(\sum_{i=0}^{n-1} a_i^2\log(\frac{2 \sum_{i=0}^{n-1} a_i^2}{\sum_{i=0}^{n-1}(a_i^2+b_i^2)})+\sum_{i=0}^{n-1} b_i^2\log(\frac{2 \sum_{i=0}^{n-1} b_i^2}{\sum_{i=0}^{n-1}(a_i^2+b_i^2)})\right).
\end{aligned}
\]
Applying the $n$--LSI to the first two terms, and $2$--LSI \eqref{Ineq: 2-LSI} with $x=\norm{a}_2$ and $y=\norm{b}_2$ to the last term, yields
\[
\mathrm{H}_{2n}[\lambda]\le \inner{a,\Gamma(n)a}+\inner{b,\Gamma(n)b}
+\frac{1}{2n}\big(\norm{a}_2-\norm{b}_2\big)^2.
\]
Finally, by Proposition~\ref{Prop: Compare Dirichlet form for n and 2n},
\begin{equation}\label{Eqn: induction from n-LSI to 2n-LSI}
  \mathrm{H}_{2n}[\lambda]\le \inner{a,\Gamma(n)a}+\inner{b,\Gamma(n)b}
+\frac{1}{2n}\big(\norm{a}_2-\norm{b}_2\big)^2
\le 2\inner{\lambda,\Gamma(2n)\lambda}.
\end{equation}\par 

For odd $n_0$, define the weight function on $\mathbb{Z}_n$ for even $n\geq n_0$
\begin{equation}\label{Eqn: definition of induced gamma_n for odd base case}
  \gamma_n(k)=
  \begin{cases}
    \psi_n(k), & k\ne \tfrac{n}{2},\\
    1,           & k=\tfrac{n}{2}.
  \end{cases}
\end{equation}
The pair $(\psi_{n_0},\gamma_{2n_0})$ satisfies \eqref{Cond: gamma pair condition} and \eqref{Ineq: Quadratic inequality from n to 2n}. For even $n$, the pair $(\gamma_n,\gamma_{2n})$ also satisfies \eqref{Cond: gamma pair condition}, and the desired inequality \eqref{Ineq: Quadratic inequality from n to 2n} becomes
\begin{equation}\label{Ineq: quadratic for odd n}
  f(x)\coloneqq r_b (n-3) x^2+x \left(2 \sqrt{r_a+1} \sqrt{r_b+1}-2\right)+r_a (n-3)\geq 0
\end{equation}
for all $x\ge 0$ and $0\le r_a,r_b\le 1$. Since $n-3>0$, the minimum of $f$ is attained at $x=-\frac{2 \sqrt{r_a+1} \sqrt{r_b+1}-2}{2r_b(n-3)}\leq 0$ and the value of the minimum of $f$ on $[0,\infty)$ is $f(0)=r_a(n-3)\geq 0$. Hence \eqref{Ineq: quadratic for odd n} holds. Therefore Proposition~\ref{Prop: Compare Dirichlet form for n and 2n} applies to the pairs $(\gamma_n,\gamma_{2n})$ defined by \eqref{Eqn: definition of induced gamma_n for odd base case}. Iterating \eqref{Eqn: induction from n-LSI to 2n-LSI} along $\big((\psi_{n_0},\gamma_{2n_0}),(\gamma_{2n_0},\gamma_{4n_0}),\ldots\big)$, we obtain for all $m\ge 0$,
\[
  \mathrm{H}_{n_0\cdot 2^{m+1}}[\lambda]
  \le 2\inner{\lambda,\Gamma(n_0\cdot 2^{m+1})\lambda}.
\]
Let $\Psi(n)=\frac{1}{n}F_n\diag(\psi_n)F_n^{-1}$. Since $\gamma_n\leq \psi_n$ pointwise, we have
\[
  \inner{\lambda,\Gamma(n_0\cdot 2^{m+1})\lambda} 
  \le \inner{\lambda,\Psi(n_0\cdot 2^{m+1})\lambda}.
\]
Therefore,
\[
  \mathrm{H}_{n_0\cdot 2^{m+1}}[\lambda]
  \le 2\inner{\lambda,\Psi(n_0\cdot 2^{m+1})\lambda}.
\]
By Lemma~\ref{Lem: Explicit form of Log Sobolev ineq}, the \(n_0\cdot 2^{m+1}\)--LSI holds for 
\(\psi_{n_0\cdot 2^{m+1}}\) for all \(m\ge 0\), which proves the claim.
\end{proof}

We now deduce Theorem~\ref{Thm: Log Sobolev inequality n=6 times 2^k and n=8 times 2^k} from Theorem~\ref{Thm: induction from n-LSI to 2n-LSI} by applying the latter with the base weights $\phi_4$ and $\phi_6$ and an appropriate family of auxiliary weight pairs $(\gamma_n,\gamma_{2n})$.

\begin{proof}[Proof of Theorem~\ref{Thm: Log Sobolev inequality n=6 times 2^k and n=8 times 2^k}]
  \textbf{(1) Case $n=3\cdot 2^m$ with $m\geq 1$.} Let $n$ be an even integer with $n\ge 6$. Define the weight function on $\mathbb{Z}_n$ by
\begin{equation}\label{Eqn: defination of induced gamma_n}
  \gamma_n(k)=
  \begin{cases}
    \psi_n(k), & k\ne \tfrac{n}{2},\\
    \tfrac{n}{2}-1, & k=\tfrac{n}{2}.
  \end{cases}
\end{equation}
We would like to apply Theorem~\ref{Thm: induction from n-LSI to 2n-LSI}. To this end, we note that the pair $(\gamma_n,\gamma_{2n})$ satisfies \eqref{Cond: gamma pair condition}, and the desired inequality \eqref{Ineq: Quadratic inequality from n to 2n} becomes
\[
  f(x)\coloneqq (r_b+n-2)x^2+x\Big(2\sqrt{r_a+1}\sqrt{r_b+1}-2n+2\Big)
  +r_a+n-2 \ge 0.
\]
Since $r_b+n-2>0$, the minimum of $f$ is attained at $x=-\frac{2\sqrt{r_a+1}\sqrt{r_b+1}-2n+2}{2(r_b+n-2)}$ and the value of the minimum is
\[
  \frac{2(n-1)\big(\sqrt{r_a+1}\sqrt{r_b+1}-1\big)
        +(n-3)\big(r_a+r_b\big)}{r_b+n-2},
\]
which is nonnegative for $0\le r_a,r_b\le 1$. Hence the pair $(\gamma_n,\gamma_{2n})$ defined in \eqref{Eqn: defination of induced gamma_n} satisfies \eqref{Ineq: Quadratic inequality from n to 2n}.\par

By Theorem~\ref{Thm: Log Sobolev inequality n=6}, the $6$--LSI holds with weight $\phi_6$; since $\phi_6\le \gamma_6\le \psi_6$, the $6$--LSI holds for $\gamma_6$ as well. Repeated application of Theorem~\ref{Thm: induction from n-LSI to 2n-LSI} yields the $6\cdot 2^m$--LSI holds for $\gamma_{6\cdot 2^m}$ for all $m\ge 1$. Finally, because $\psi_n\ge \gamma_n$ pointwise,
\[
  \mathrm{H}_n[\lambda]  \le 2\inner{\lambda,\Gamma(n)\lambda}
  \le 2\inner{\lambda,\Psi(n)\lambda},
\]
and by Lemma~\ref{Lem: Explicit form of Log Sobolev ineq} the $6\cdot 2^m$--LSI holds for $\psi_{6\cdot 2^m}$.\par

\textbf{(2) Case $n=2^m$ with $m\geq 1$.} The case $n=2$ is classical  and the case $n=4$ is due to the work \cite{MR730056} and Gross's extrapolation technique~\cite{MR420249}. Note that the pair $(\phi_4,\gamma_8)$ satisfies \eqref{Cond: gamma pair condition}, and \eqref{Ineq: Quadratic inequality from n to 2n} becomes
\[
  f(x)\coloneqq \left(2-\frac{r_b}{5}\right)x^2+2x\left(\sqrt{r_a+1}\sqrt{r_b+1}-3\right)
  -\frac{r_a}{5}+2\ \ge 0,
\]
with $0\le r_a,r_b\le 1$ and $x\ge0$. As $2-\frac{r_b}{5}>0$, the minimum of $f$ on $[0,\infty)$ is attained at $x=-\tfrac{2 \left(\sqrt{r_a^2+1} \sqrt{r_b^2+1}-3\right)}{2\left( 2-r_b^2/5 \right)}$ and the value of the minimum is
\[
  \frac{r_a(24r_b+35)+5\left(-30\sqrt{r_a+1}\sqrt{r_b+1}+7r_b+30\right)}
       {5(r_b-10)}
  =\frac{h(r_a,r_b)}{5(r_b-10)}.
\]
Since $r_b\in[0,1]$, the denominator is negative. In order to show \eqref{Ineq: Quadratic inequality from n to 2n} hold, it suffices to show $h(r_a,r_b)\le 0$ on $[0,1]^2$. A direct computation gives
\[
  \pdv[2]{h}{r_a}(r_a,r_b)=\frac{75\sqrt{r_b+1}}{2(r_a+1)^{3/2}}>0,
  \qquad
  \pdv[2]{h}{r_b}(r_a,r_b)=\frac{75\sqrt{r_a+1}}{2(r_b+1)^{3/2}}>0.
\]
Thus $h$ is convex in each variable separately. So the maximum of $h$ on $[0,1]^2$ is attained at a corner. Evaluating,
\[
  h(0,0)=0,\quad h(0,1)=h(1,0)=5\big(37-30\sqrt{2}\big)<0,\quad  h(1,1)=-56<0.
\]
Hence $\max_{[0,1]^2} h\le 0$, and since $5(r_b-10)<0$ we conclude
\[
  \frac{r_a(24r_b+35)+5\left(-30\sqrt{r_a+1}\sqrt{r_b+1}+7r_b+30\right)}
       {5(r_b-10)} \ge 0.
\]
Therefore $(\phi_4,\gamma_8)$ satisfies \eqref{Ineq: Quadratic inequality from n to 2n}. Recall that $(\gamma_{8\cdot 2^m},\gamma_{8\cdot 2^{m+1}})$ also satisfies \eqref{Ineq: Quadratic inequality from n to 2n}. Hence by Theorem~\ref{Thm: Log Sobolev inequality n=4} and Theorem~\ref{Thm: induction from n-LSI to 2n-LSI} the $8\cdot 2^m$--LSI holds for $\gamma_{8\cdot 2^m}$ for all $m\ge 0$, and therefore for $\psi_{8\cdot 2^m}$.
\end{proof}

\begin{remark}
  The role of $\phi_4$ and $\phi_6$ from Section \ref{Sec: Log Sobolev inequality n=4 and n=6} is essential for the above proof. Indeed, the original pairs $(\psi_4,\psi_8)$ and $(\psi_6,\psi_{12})$ do not satisfy \eqref{Ineq: Quadratic inequality from n to 2n}, so Theorem~\ref{Thm: induction from n-LSI to 2n-LSI} would not apply to these pairs if we replace $\phi_n$ by $\psi_n$ in the above arguments.
\end{remark}

\section*{Acknowledgments}
The author would like to thank Professor Quanhua Xu and Professor Simeng Wang for their patience and encouragement, as well as for their careful reading of the manuscript and many helpful discussions. The author is partially supported by the NSF of China (No. 12031004, No. W2441002, No. 12301161, No.12371138).

\printbibliography
\addresseshere

\end{document}